\documentclass[12pt]{amsart}

\usepackage{amssymb}
\usepackage{amsmath}
\usepackage{amscd}
\usepackage{float}
\usepackage{dirtytalk}
\usepackage{graphicx}
\graphicspath{{./}}

\usepackage{amsfonts}
\usepackage{pb-diagram}
\usepackage{eufrak}
\graphicspath{{./images/}}
\newtheorem{thm}{Theorem}[section]
\newtheorem{cor}[thm]{Corollary}

\newtheorem{prop}[thm]{Proposition}

\newtheorem{rem}{Remark}[section]

\title{Knots in $\mathbb{R}P^3$}
\author{Louis H. Kauffman, Rama Mishra and Visakh Narayanan}

\begin{document}
\maketitle
\noindent {\bf Abstract:}
This paper studies knots in three dimensional projective space. Techniques in virtual knot theory are applied to obtain a Jones polynomial for projective links and it is shown that this is equivalent to the Jones polynomial defined by Drobotukhina. A Khovanov homology theory for projective knots is constructed by using virtual Khovanov homology and the virtual Rasmussen invariant of Dye, Kaestner, and Kauffman. This homology theory is compared with the Khovanov theory developed by Manolescu and Willis for projective knots. It is shown that these theories are essentially equivalent, giving new viewpoints for both methods. The paper ends with problems about these approaches and an example of multiple projectivizations of the Figure-8 knot whose equivalence is unknown at this time. \\

\noindent \textbf{Mathematics subject classification: 57K10, 57K12, 57K14} 

\section{Introduction}

Our goal is to understand the knot theory of the real projective three space, $\mathbb{R}P^3$. Since $\mathbb{R}P^3$ is a manifold similar to $S^3$, it is natural to expect that the knot theory of these manifolds are closely related. Virtual knot theory is a perturbation of classical knot theory. We perturb it further, by adding a special move, called the \say{simple flype} and obtain a virtual model for knots in projective space. Thus structures available in the virtual world can be exported to the projective space, to obtain invariants for projective knots. One of the most interesting problems in projective knot theory, is to detect whether a knot is affine. The property of a projective knot being non-affine is represented by the property of its virtual model being non-classical. Thus, some theorems like Theorem \ref{drobcor}, known about non-classical virtual knots can be used to study the non-affineness of projective knots.\\

We show that the virtual bracket polynomial can be used to construct a bracket polynomial for projective knots. By normalizing this we obtain a Jones type polynomial for projective knots. A version of Jones polynomial for projective knots was defined by Drobotukhina \cite{julia} by using the diagrammatic theory obtained by projecting a knot in $\mathbb{R}P^3$ to some projective plane and expanding crossings by the Kauffman skein relation. It turns out that both the polynomials are equivalent. \\

We  construct a Khovanov-type homology theory by using the Khovanov homology for virtual knots defined by Manturov \cite{M}. The Lee deformation and a Rasmussen invariant for this obtained by Dye, Kaestner and Kauffman \cite{DKK} descends to a Rasmussen invariant for projective knots. Hence we have theorems like Theorem \ref{ourthm}, regarding the equality of the 4-ball genus and the Seifert genus of positive knots in $\mathbb{R}P^3$. Manolescu and Willis \cite{Manolescu} defined a Khovanov homology and a Rasmussen invariant for projective knots with very different methods. They  proved a theorem establishing the equivalence of 4-ball genus and the Seifert genus for positive knots. It is natural to wonder how these two cohomology theories are related. We  provide a structural comparison between the two and we show that they are both essentially equivalent. Thus most of the known results about projective knots can be seen as coming from this virtual model for them. \\

\noindent \textbf{Organization of the paper:} In Section \ref{flypesection} we introduce the virtual model for projective knots by adding an extra move, called the \say{flype move} to the equivalence of usual virtual knots. Then we prove that the bracket polynomial for usual virtual knots is invariant under the flype move (Theorem 2.4). Then a theorem about affine knots is proved  as Theorem 2.5 which results in Corollary \ref{drobcor} which gives an obstruction for a projective knot to be affine. In Section \ref{examples} we provide some examples showing the application of Corollary \ref{drobcor}. Section \ref{virtualknots} and Section \ref{stableclass} provide some details about virtual knot theory which will be used extensively in the later parts of the paper. Section \ref{Khovanov} defines the Khovanov-type invariant for projective knots. Theorem \ref{ourthm} is proven for projective knots. Then in \ref{comparison} we compare our theory with that of \cite{Manolescu} and show that they are equivalent. In Section 7 we discuss some properties of projective knots, which our virtual model cannot capture. We add an Appendix as Section 8, which includes the $Mathematica$ code for bracket calculation of a 13 crossing knot appearing in Section 6.

\section{Knots in projective space, virtual knots and the bracket polynomial} \label{flypesection}

 Knots in projective space $RP^3$ can be represented by diagrams in $RP^2$ where the projective plane $RP^2$  is represented by a disk with antipodal identifications on the boundary. Thus a projective knot can be represented by a diagram showing a tangle in a disk with marked tangle-ends that correspond to the antipodal identifications on the boundary of the disk. The diagram on the left of Figure \ref{class0knot} is an example. Moves on these diagrams consist in the Reidemeister moves plus two slide moves across the boundary identification as shown in Figure \ref{slide}. Since the fundamental group of $\mathbb{R}P^3$ is $\frac{\mathbb{Z}}{2\mathbb{Z}}$, there are two homotopy classes of knots in $\mathbb{R}P^3$. We call them \say{class-0} and \say{class-1} knots, where class-0 refers to the trivial homotopy class. The class of a knot is the same as the $mod2$ class of half the number of points on the boundary of the disk in any diagram of the knot. If a knot has a diagram contained entirely in the interior of a disk we call them \say{affine}. It is clear that affine knots are contained in a ball and are contractible. Hence all affine knots are class-0.\\

\par But not all class-0 knots are affine. The knot shown in Figure \ref{class0knot} cannot have a diagram which is entirely contained in the interior of the disk. A quick way to see this is that, the lift of any projective plane under the canonical double covering is a 2-sphere in $S^3$. The boundary circle of the disk is the projection of a projective plane. If a knot $K$ is affine, then it is disjoint from this projective plane in $\mathbb{R}P^3$, then the lift $K$ under the covering map is a link with two components which are separated by the sphere covering this projective plane. Hence the covering link should have zero linking number. It is easy to see that the covering link of the knot in Figure \ref{class0knot} has linking number 2. Thus there are class-0 knots which are \say{non-affine}. Note that all class-1 knots are non-affine.Thus there are three families of knots in $\mathbb{R}P^3$ namely affine, class-0 non-affine and class-1. An interested reader may consult \cite{MN} for more details.

\begin{figure}
\centering
\includegraphics{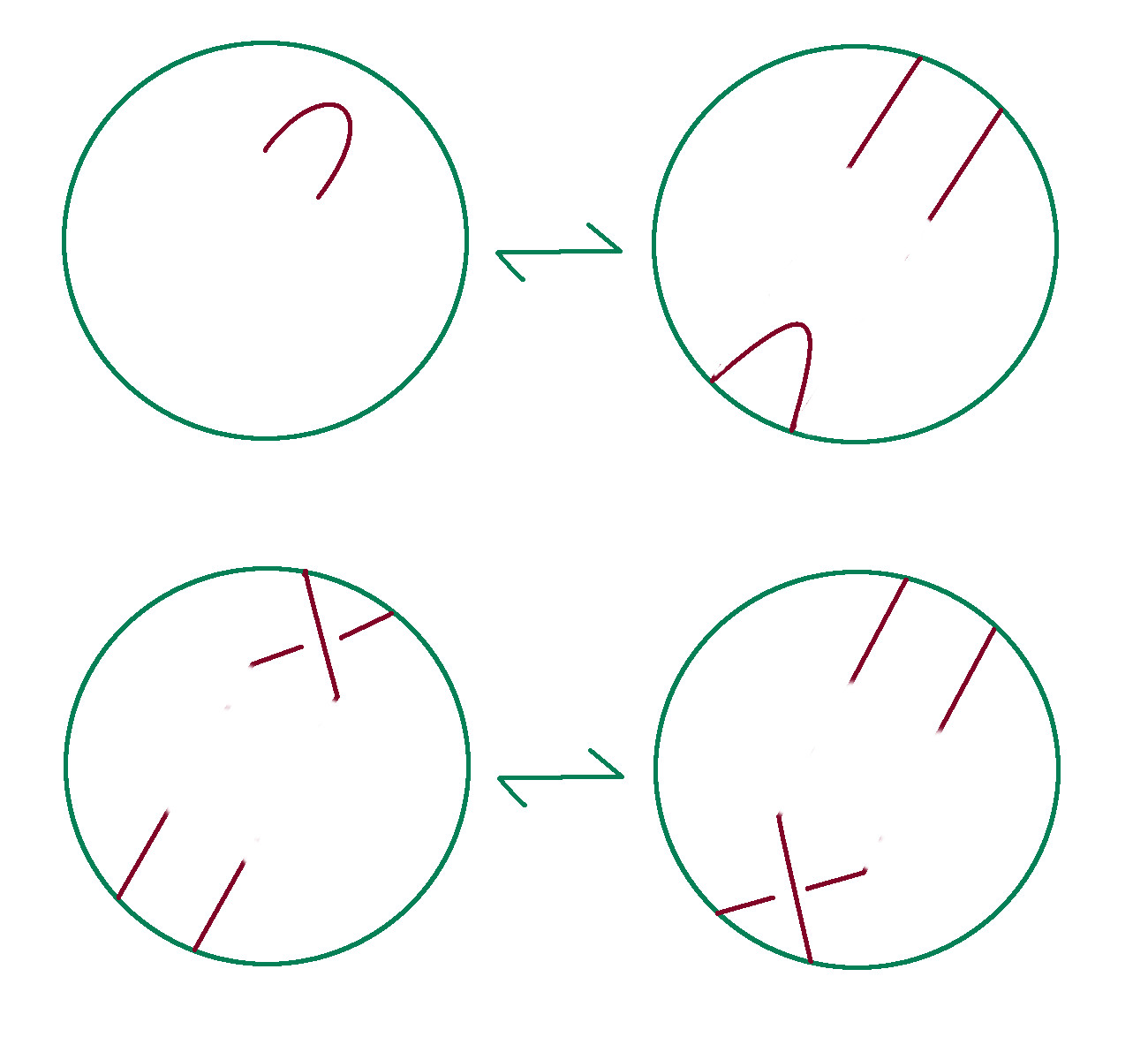}
\caption{Slide Moves}
\label{slide}
\end{figure}

Refer to Figure \ref{slide} and Figure \ref{crosslide}, to note that when a crossing slides across the boundary its oriented sign remains the same. We are viewing the projection of the diagram to the projective plane in terms of the embedding of a thickened projective plane in $RP^{2} \subset RP^{3}.$ The neighbourhood of the projective plane in the projective three space is a twisted $I-$ bundle over the projective plane. One can visualize the slide move by imagining sliding a diagram drawn on a M\"obius strip through the twist in the strip. Because we respect the twisted $I-$bundle, the viewer of the diagram twists as he/she moves through the twist. This means that the crossing sign is retained as shown in Figure \ref{crosslide}.

\begin{figure}
\centering
\includegraphics[scale=0.8]{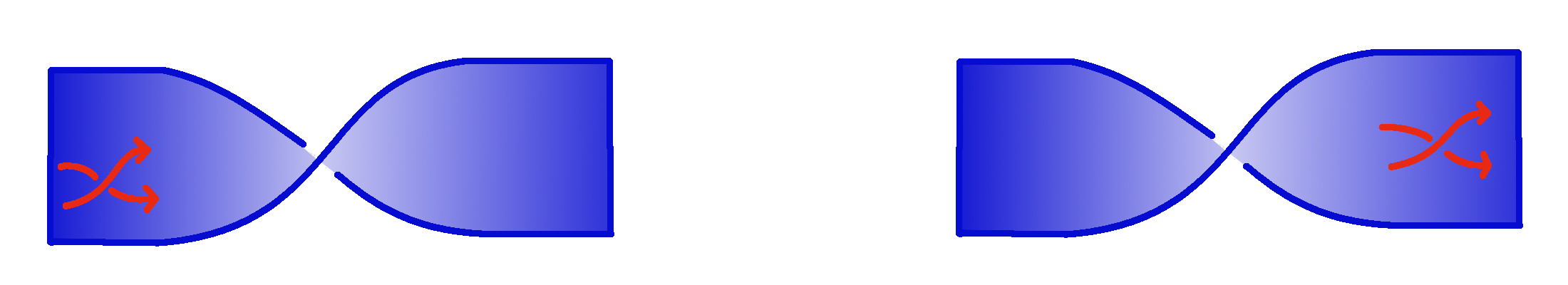}
\caption{Sliding a crossing over a half twist}
\label{crosslide}
\end{figure}

Since the  Reidemeister and slide moves are compatible with crossing types as we have discussed above, one can define the bracket polynomial expansion and its writhe normalization just as we have done in the case of diagrams in the plane or on the two dimensional sphere. This means that the loop counting for the bracket polynomial can be done by representing the knot in projective space by a tangle drawn in the projective disk and identifying the endpoints of this tangle according to the antipodal identifications on the boundary of the disk. This identification can be simplified by noting that for n-strands in the tangle top and bottom) there is a canonical permutation corresponding to the antipodal map. The tangle may not have the same number of endpoints at the top and bottom in which case we can still use the antipodal identification.\\

\begin{figure}
\centering
\includegraphics[scale=1]{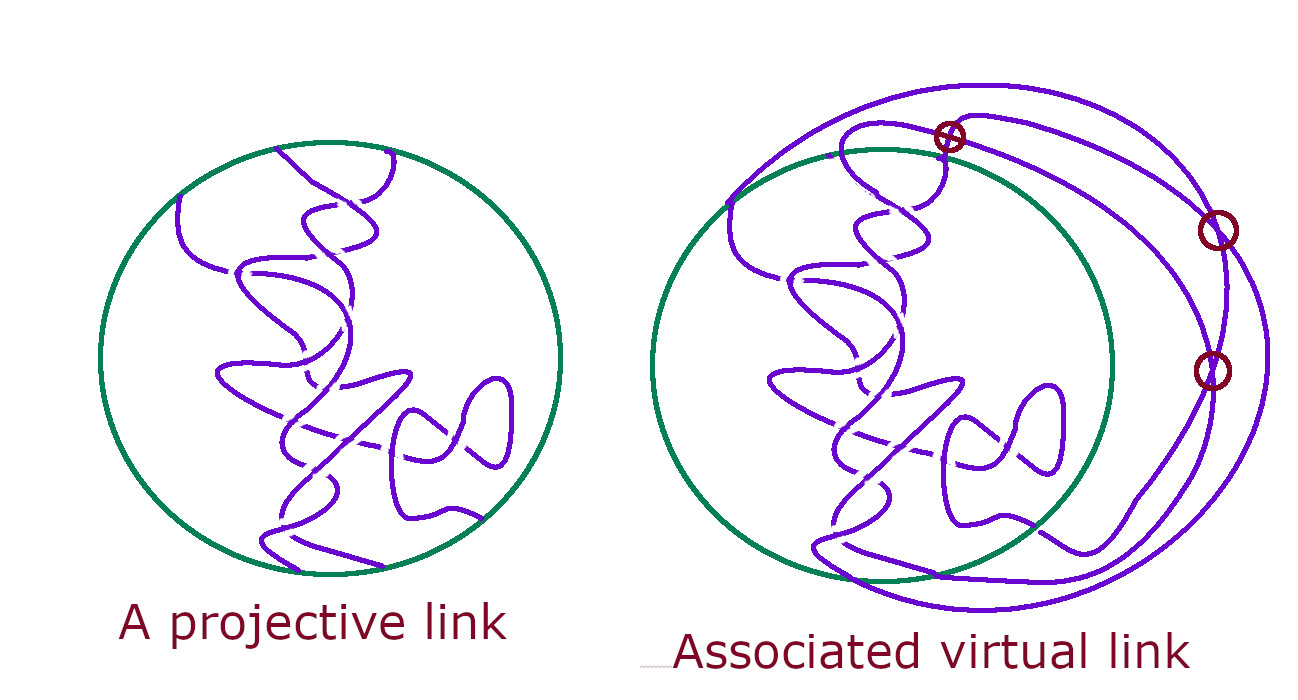}
\caption{Association of a virtual link to a projective link.}
\label{pi}
\end{figure}

We define a mapping $\pi$ from diagrams in the projective plane to virtual knot and link diagrams. See Figure \ref{pi}. Diagrammatically the association of a virtual link diagram to a tangle representation of a projective knot is this: Regard the disk for the projective tangle as embedded in the plane or two-sphere. Draw the connecting arcs between identified points on the boundary of the disk as actual arcs embedded in the complement of the disk in the plane. The arcs may cross transversally and these crossings are taken as virtual crossings. If $T$ is the tangle associated with the projective link, let $\pi(T)$ denote the corresponding virtual link diagram we have just described. We will now prove that if $K$ and $L$ are equivalent projective links, then $\pi(K)$ and $\pi(L)$ are related by a simple set of moves on virtual links. Thus certain invariants of virtual links can be used to construct projective link invariants.\\

Recall that the bracket polynomial for virtual knots, is defined by the usual rules, shown in Figure \ref{bracketdef}. 
\begin{figure}
\includegraphics{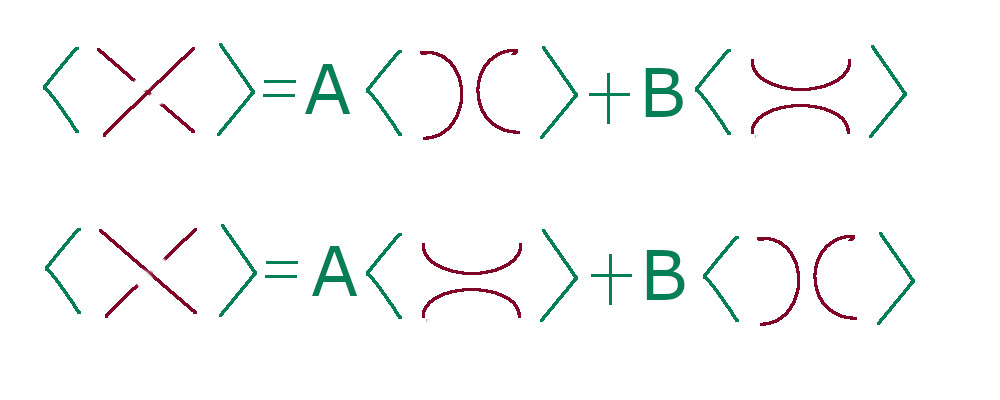}\\
\includegraphics{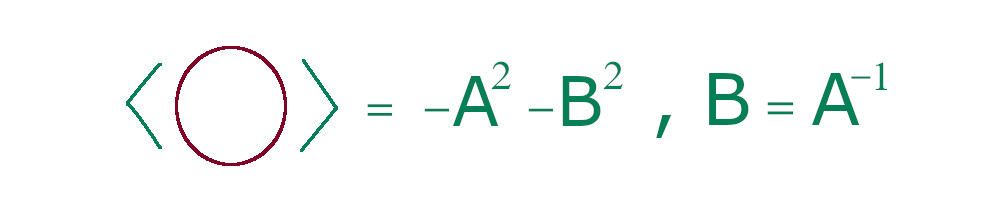}
\caption{Defining rules for the bracket polynomial.}
\label{bracketdef}
\end{figure}

 In the case of the bracket polynomial we shall show that $\langle K \rangle  = \langle \pi(K) \rangle $ for any projective link $K$ where the first bracket is the bracket as defined directly on the projective diagram and the second bracket is the usual simplest bracket polynomial for virtual links. By using this mapping of projective links to virtual links, we can apply invariants of virtual links to measure properties of links in projective space. We will give a number of examples of this correspondence in the course of the paper.\\
 
\begin{figure}
\centering
\includegraphics{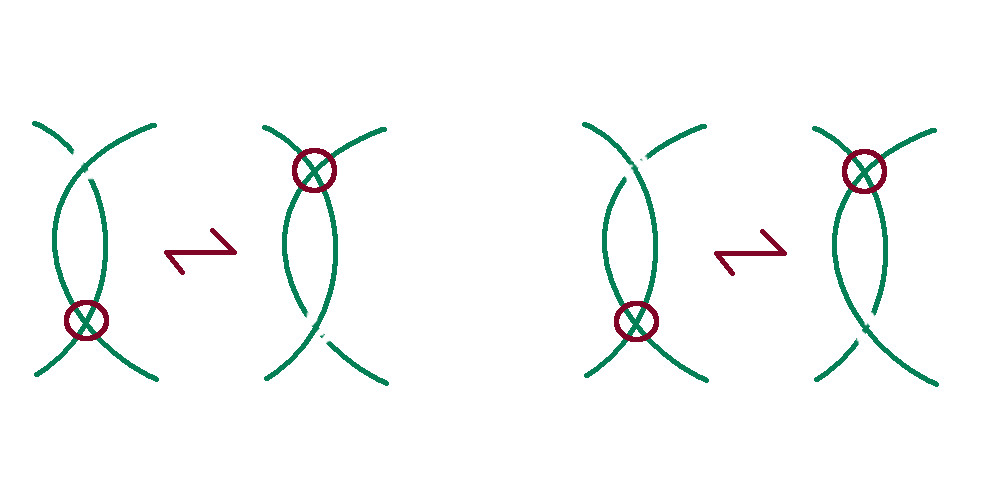}
\caption{Flype move}
\label{flype}
\end{figure}

The scheme described above will associate a diagram of a virtual link with the diagram of a projective link. Now we study this association in more detail. It is easy to see that, if we perform a slide move by transferring a crossing across the projective plane at infinity, the associated virtual link before and after are related by the move shown in Figure \ref{flype}. We call this as the \say{Flype move}. We prove the following.

\begin{prop}
Any two virtual link diagrams associated with a link in $\mathbb{R}P^3$, are related by a sequence of ordinary virtual Reidemeister moves \cite{vkt} and the Flype move. \\

\end{prop}
\begin{proof} Let $L$ and $L’$ be projective links that are equivalent in $\mathbb{R}P^3$. Let $\pi(L)$ and $\pi(L’)$ be the associated virtual links. Then $\pi$(L) and $\pi(L’)$ are related by a sequence of standard virtual moves and the Flype move. The virtual link diagrams associated to the diagrams of a projective link before and after performing any Reidemeister move on a projective link diagram is clearly related by the corresponding generalized Reidemeister move on virtual link diagrams. Hence it is enough to examine how the associated virtual links are related before and after performing each of the slide moves shown in Figure \ref{slide}. It is clear the virtual link diagrams associated before and after performing a slide move-I are related by a virtual Reidemeister move-I. Similarly, the virtual link diagrams associated before and after a Slide move-II are related by a Flype move. Since any two diagrams of a projective link are related by a finite sequence of Reidemeister moves and slide moves \cite{julia}, the proposition follows. 
\end{proof}

Let PVK denote the set of virtual links up to usual Reidemeister moves and the Flype move. Then by Proposition 1, we have that, the association defined above will induce a mapping, 

\begin{equation*}
\pi: \{Links\ in\ projective\ space/ambient\ isotopy\}\to PVK. 
\end{equation*}

\vspace{0.7cm}

\begin{cor}
If K is an affine link in $\mathbb{R}P^3$, then $\pi(K)$ is a classical virtual knot.
\end{cor}

\begin{cor}
Any invariant of virtual link diagrams that is invariant under the flype move is an invariant of links in $\mathbb{R}P^3$.
\end{cor}

\begin{rem}
We have elsewhere\cite{vkt} called the equivalence relation of virtual links generated by standard moves and the Flype move, Z-Equivalence. Thus projective links that are equivalent in $\mathbb{R}P^3$ map to Z-equivalent virtual links. In this paper we will only say virtual links equivalent under the Flype move and not use the term Z-equivalence. 
\end{rem}

\begin{thm}\label{flypeinv}
The normalized Kauffman bracket polynomial is invariant under the flype move.  
\end{thm}
\begin{figure}
\centering
\includegraphics{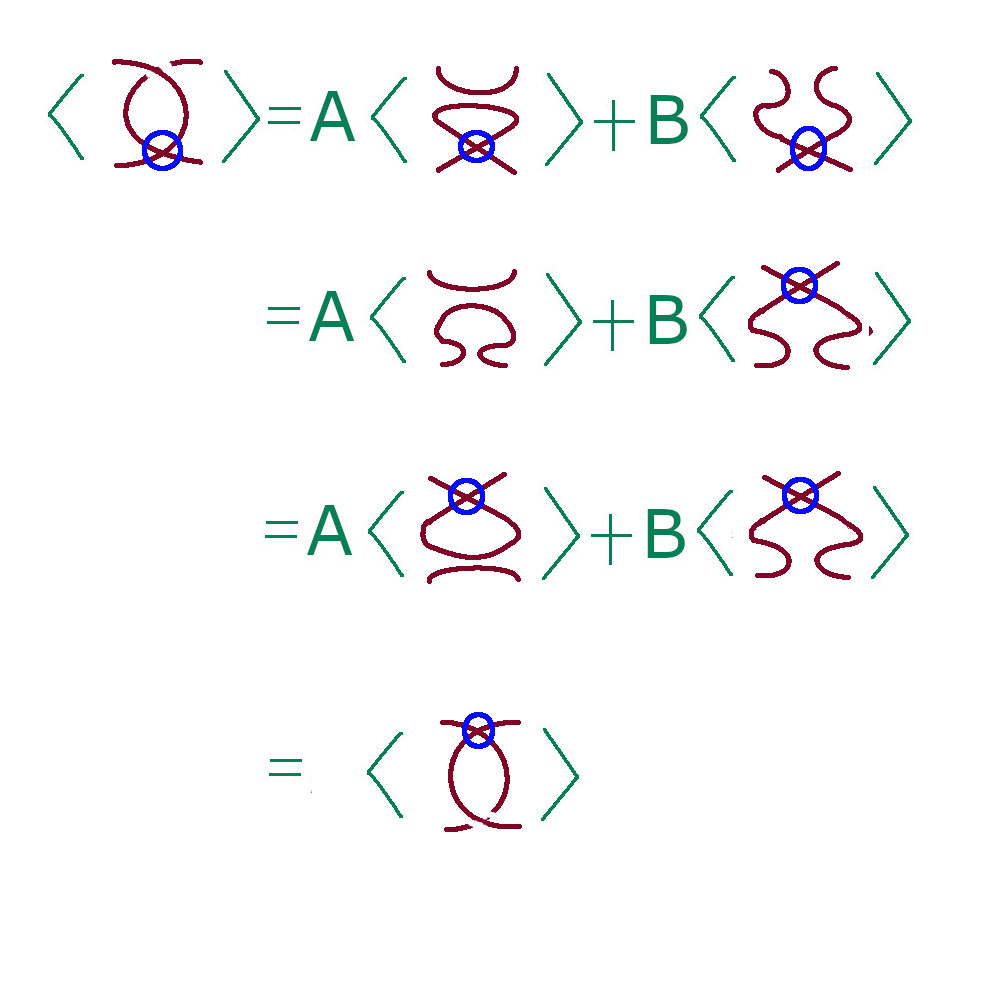}
\caption{Invariance of the bracket polynomial under the flype move.}
\label{bracketflype}
\end{figure}

\begin{proof} Refer to Figure \ref{bracketflype}. \end{proof}

\vspace{0.5cm}

\par Thus we may use the Kauffman bracket polynomial on virtual links to construct invariants of links in $\mathbb{R}P^3$. Consider the normalized bracket polynomial on virtual knot diagrams defined by setting $B=A^{-1}$ and normalizing by dividing with the loop value $-A^2-A^{-2}$. The normalized bracket of a virtual link with writhe w(K) is given by the formula,
\begin{equation*}
f_K{(A)}= \frac{-A^{-3w(K)}}{(-A^2-A^{-2})}\langle K\rangle,
\end{equation*}
a Laurent polynomial of the variable $A$.  When we change variables from $A$ to $t$ by the rule, $A=t^{-\frac{1}{4}}$ we get a Laurent polynomial in variable $t$. This is the Jones polynomial. In this paper we call $f_K$ the normalized bracket polynomial, as a function of the variable $A$.\\

\par Now for a knot $K$ in $\mathbb{R}P^3$, if $\pi(K)$ is any virtual link associated to $K$ using a diagram as described above, we define,
\begin{equation*}
f_K := f_{\pi(K)}.
\end{equation*}
 Notice that by Theorem 1, $f_K$ is well defined. In what follows we will use this notation quite frequently. The following theorem is from \cite{julia}.

\begin{thm}
If K is an affine knot in $\mathbb{R}P^3$, then the degree of each monomial in $f_K$ is an integer multiple of 4. 
\end{thm}
\begin{proof} Note that if $J$ is a classical knot then the degree of each monomial in $f_J$ is an integer multiple of 4. Since $K$ is an affine knot it has a diagram such that the associated tangle $T$ has no boundary points. Hence there are no virtual crossings in $\pi(T)$. Let $J$ be a classical knot obtained by removing a projective plane disjoint from $K$. Notice that $f_K  := f_{\pi(T)}   = f_J   $. Hence all the monomials in $f_K $ have degree as integer multiples of 4. 
\end{proof}

\begin{cor}\label{drobcor}
If K is a knot in $\mathbb{R}P^3$ and $f_K  $ has a monomial with degree not divisible by 4, then $K$ is non-affine.
\end{cor}

\section{Examples} \label{examples}

\begin{figure}
\centering
\includegraphics[scale=0.8]{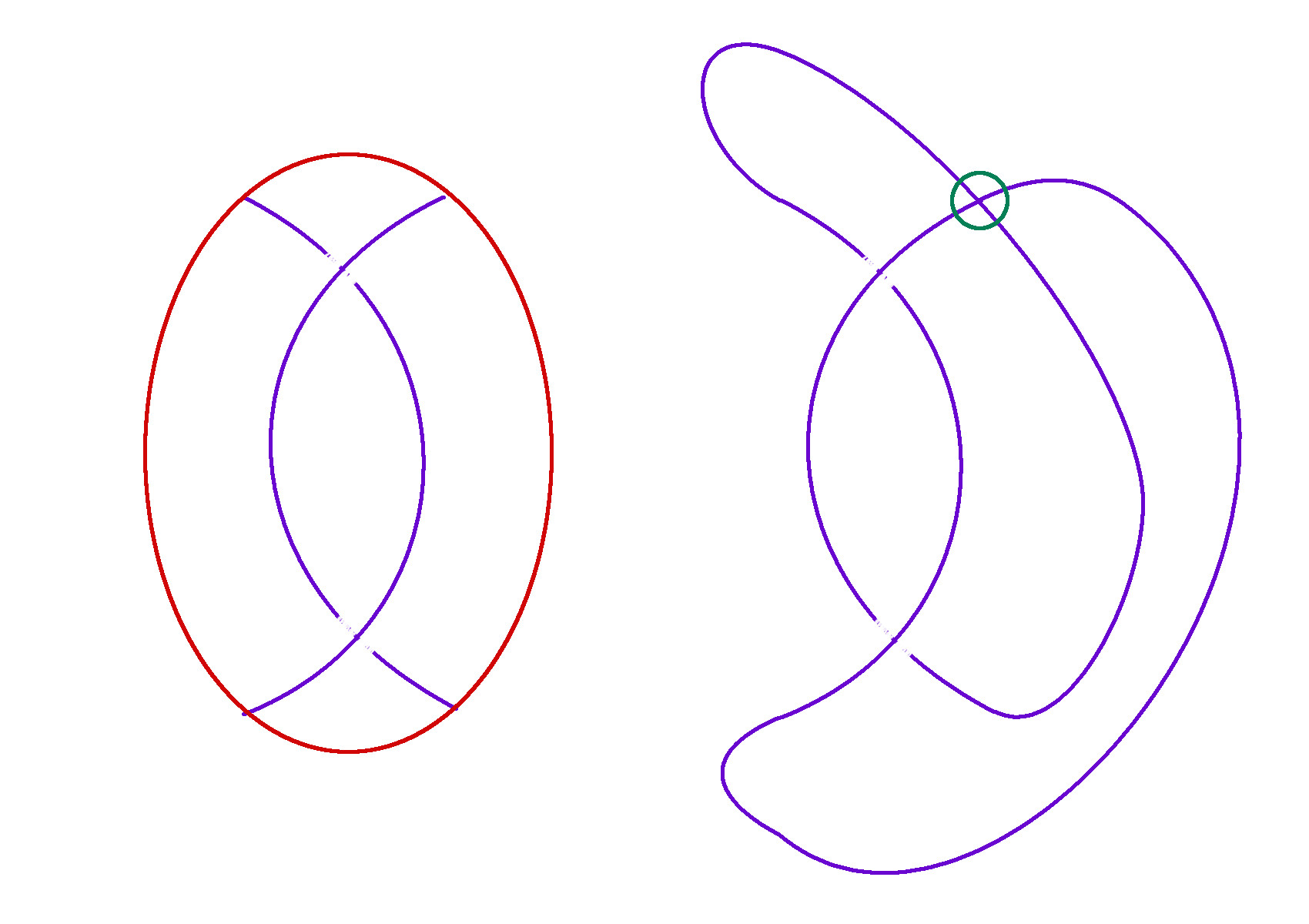}
\caption{The virtual knot associated with a class-0 knot.}
\label{class0knot}
\end{figure}

Let $J$ be the projective knot shown in the left of Figure \ref{class0knot}. It is a class-0 non-affine knot as discussed in Section 2. The virtual knot shown on the right of Figure \ref{class0knot} is constructed using the same method as above. Using this virtual knot one can easily calculate,

\begin{equation*}
f_J   = -A^{-4}-A^{-6}+A^{-10}.
\end{equation*}
Hence the normalized bracket polynomial detects that $J$ is non-affine, since two of the exponents are not divisible by 4.\\

\begin{figure}
\centering
\includegraphics[scale=0.8]{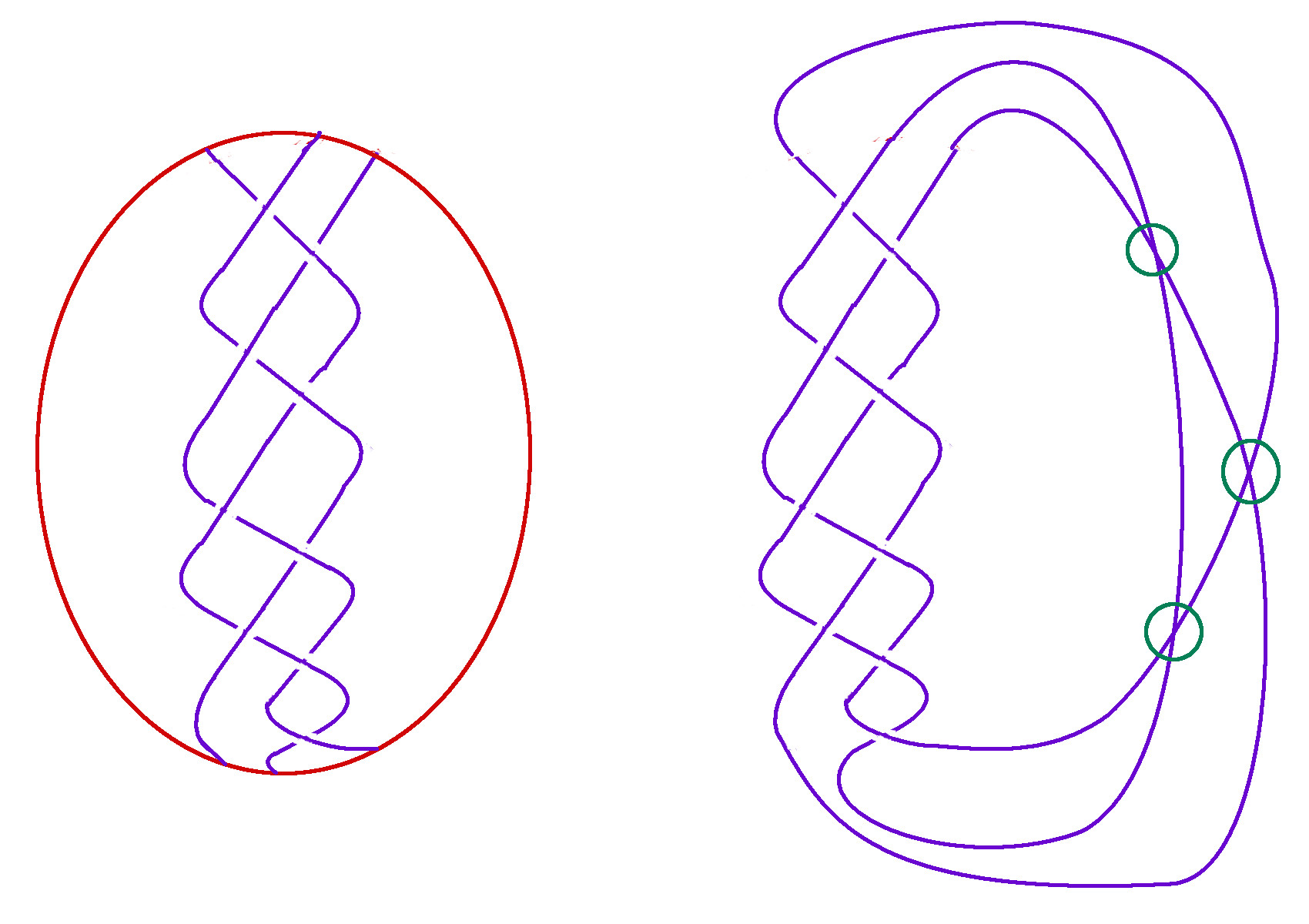}
\caption{The virtual link associated to a class-1 knot.}
\label{weave}
\end{figure}

Let $K_1$ be the knot shown in Figure \ref{weave}. We have,
\begin{equation*}
f_{K_1}   =  3+ A^{-8}-2A^{-4}-2A^{4}+2A^8-2A^{12}+A^{16}
\end{equation*}
Since the knot is class-1, it is non-affine, although the bracket does not detect it. It is noteworthy that the bracket detects the chirality of $K_1$.

\begin{figure}
\centering
\includegraphics[scale=0.8]{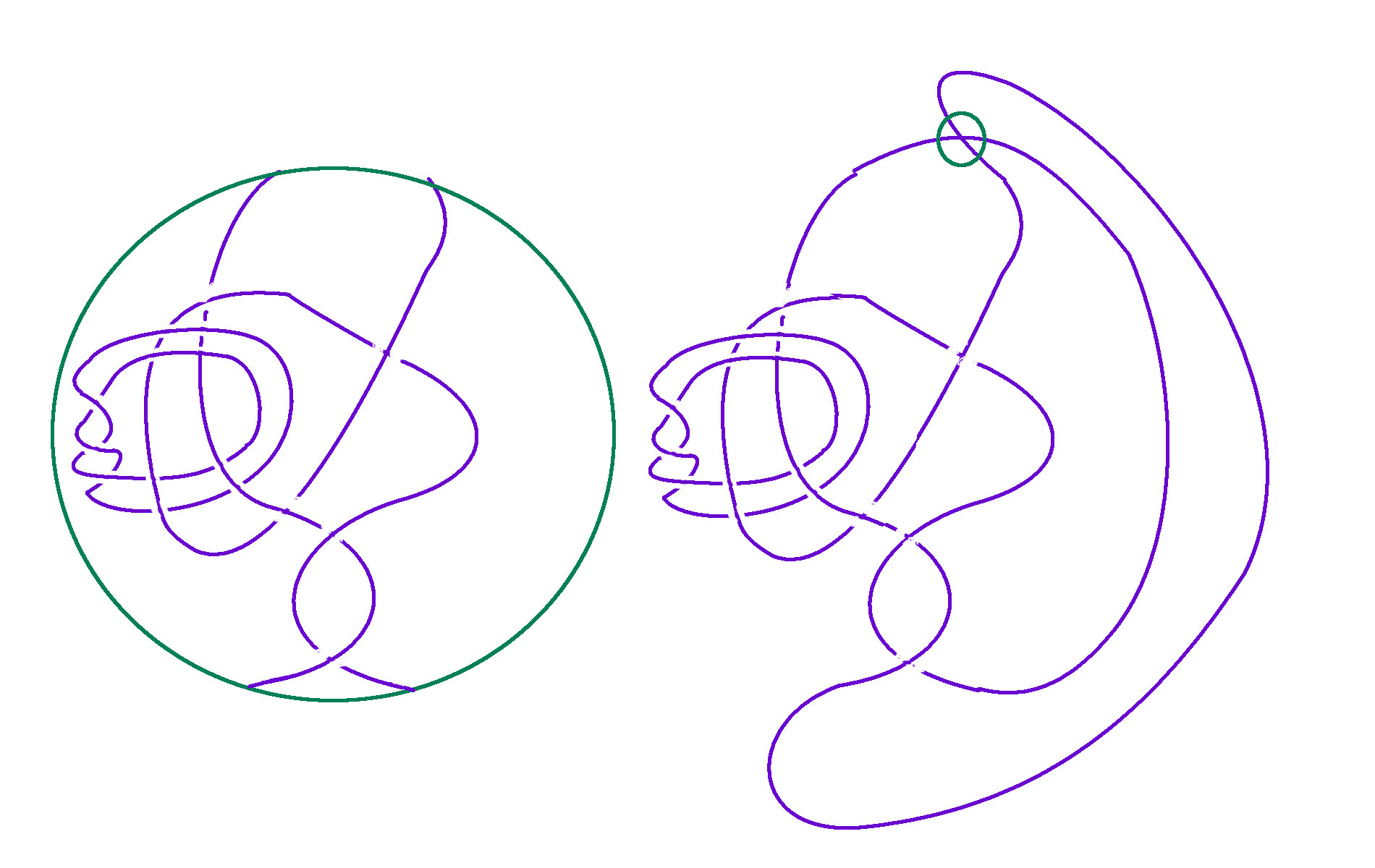}
\caption{The virtual link associated to the projective W-tangle.}
\label{wonky}
\end{figure}
Let $W_1$ be the virtual knot shown in Figure \ref{wonky}. It is a virtual knot associated to a link $K_2$ that is the projective closure of the W-tangle \cite{eli}. 

\begin{equation*}
f_{W_1}   = -A^{-8}-A^{-4}
\end{equation*} 

Thus $W_1$ has the normalized bracket of a classical knot.

\begin{figure}
\centering
\includegraphics[scale=0.7]{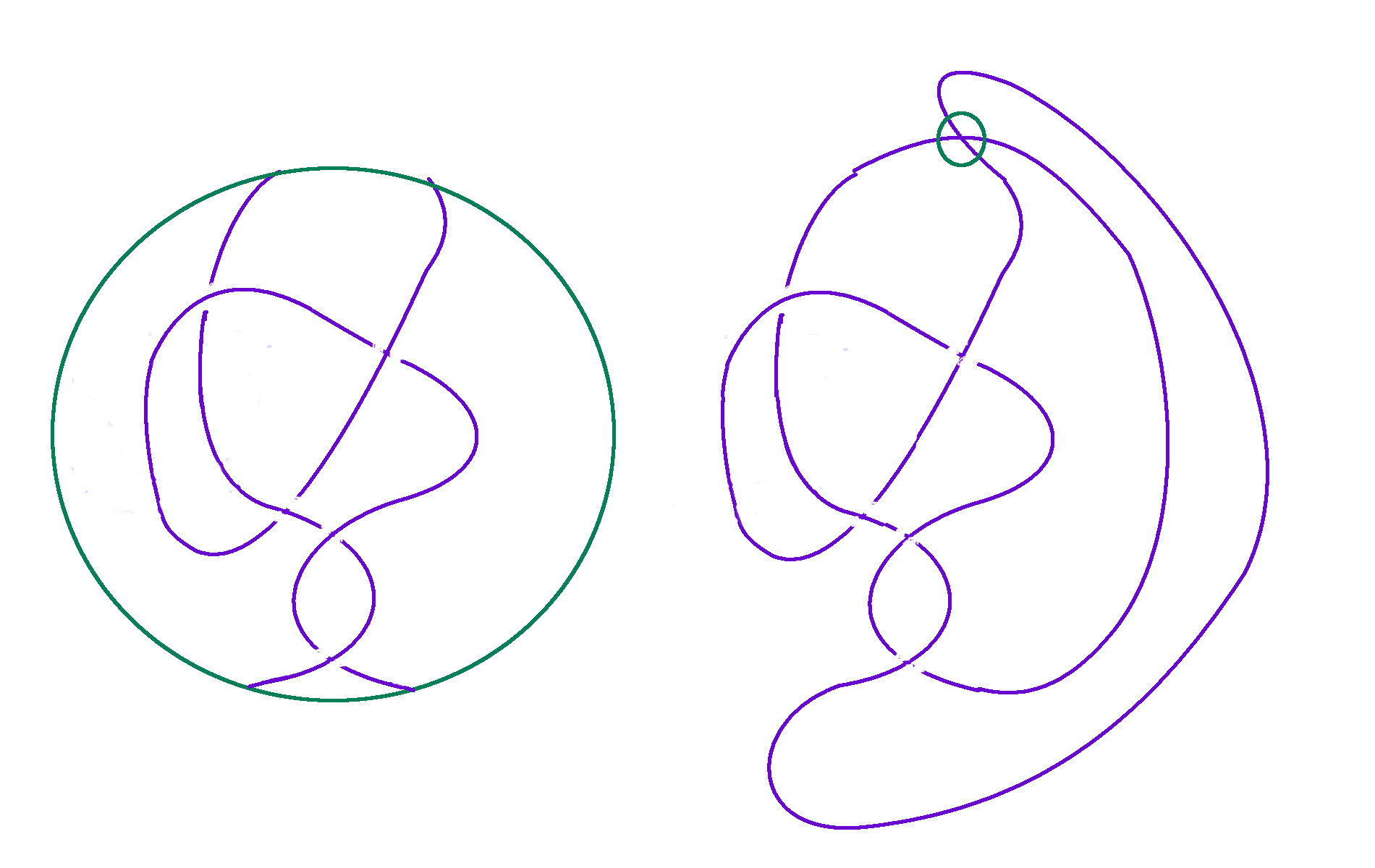}
\caption{A component of the virtual link in Figure \ref{wonky}.}
\label{wonkycomp}
\end{figure}

Now let $W_2$ be the virtual link shown in Figure \ref{wonkycomp}. It is the virtual link associated to a component of the link $K_2$, which we call $K_3$.

\begin{equation*}
f_{W_2}   = A^{-14}(1-2A^{4}+2A^{8}+A^{10}-2A^{12}-A^{14}+A^{16}+A^{18})
\end{equation*} 
Thus $K_3$ is non-affine by Corollary \ref{drobcor}. Since $K_3$ is a component of $K_2$, $K_2$ is  non-affine. 

\begin{figure}
\centering
\includegraphics[scale=0.5]{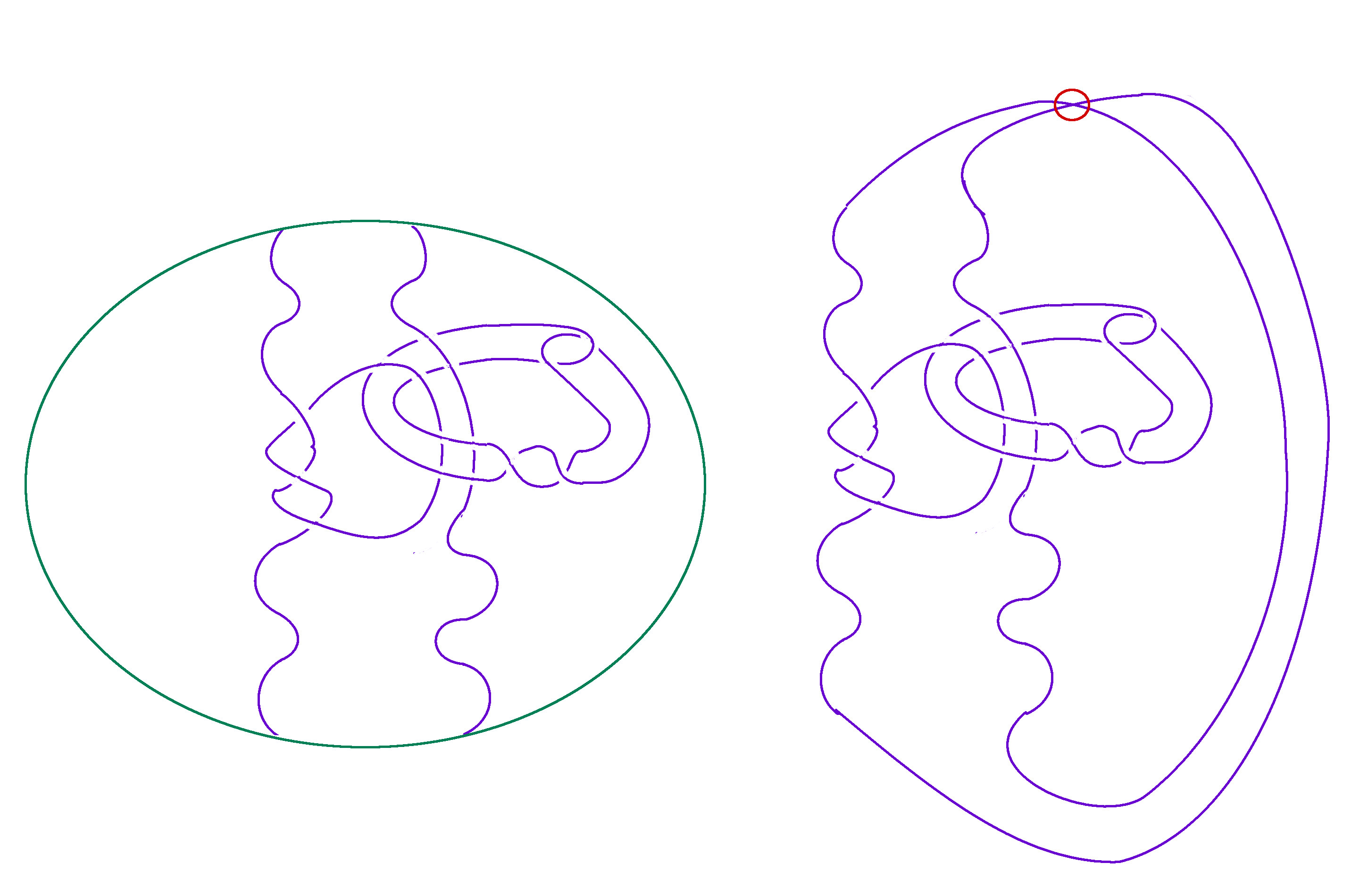}
\caption{The virtual link associated to a projective Thistlethwaite tangle.}
\label{thistle}
\end{figure}

\section{Virtual Knot Theory} \label{virtualknots}
Knot theory
studies the embeddings of curves in three-dimensional space.  Virtual knot theory studies the  embeddings of curves in thickened surfaces of arbitrary
genus, up to the addition and removal of empty handles from the surface. Virtual knots have a special diagrammatic theory, described below,
that makes handling them
very similar to the handling of classical knot diagrams. Many structures in classical knot
theory generalize to the virtual domain.
\bigbreak  

In the diagrammatic theory of virtual knots one adds 
a {\em virtual crossing} (see Figure~\ref{Figure 1}) that is neither an over-crossing nor an under-crossing.  A virtual crossing is represented by two crossing segments with a small circle
placed around the crossing point. 
\bigbreak

Moves on virtual diagrams generalize the Reidemeister moves for classical knot and link diagrams.  See Figure~\ref{Figure 1}.  One can summarize the moves on virtual diagrams by saying that the classical crossings interact with one another according to the usual Reidemeister moves while virtual crossings are artifacts of the attempt to draw the virtual structure in the plane. 
A segment of diagram consisting of a sequence of consecutive virtual crossings can be excised and a new connection made between the resulting free ends. If the new connecting segment intersects the remaining diagram (transversally) then each new intersection is taken to be virtual. Such an excision and reconnection is called a {\it detour move}. Adding the global detour move to the Reidemeister moves completes the description of moves on virtual diagrams. In Figure~\ref{Figure 1} we illustrate a set of local moves involving virtual crossings. The global detour move is a consequence of  moves (B) and (C) in Figure~\ref{Figure 1}. The detour move is illustrated in Figure~\ref{Figure 2}.  Virtual knot and link diagrams that can be connected by a finite sequence of these moves are said to be {\it equivalent} or {\it virtually isotopic}.
\bigbreak

\begin{figure}[htb]
     \begin{center}
     \begin{tabular}{c}
     \includegraphics[width=10cm]{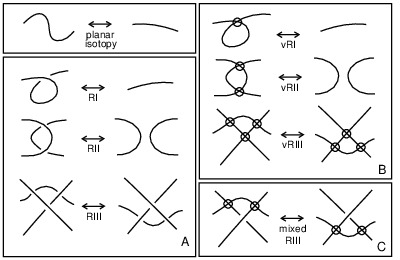}
     \end{tabular}
     \caption{\bf Moves}
     \label{Figure 1}
\end{center}
\end{figure}

\begin{figure}[htb]
     \begin{center}
     \begin{tabular}{c}
     \includegraphics[width=10cm]{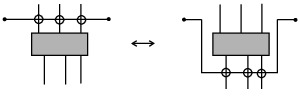}
     \end{tabular}
     \caption{\bf Detour Move}
     \label{Figure 2}
\end{center}
\end{figure}

\begin{figure}[htb]
     \begin{center}
     \begin{tabular}{c}
     \includegraphics[width=10cm]{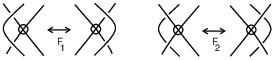}
     \end{tabular}
     \caption{\bf Forbidden Moves}
     \label{Figure 3}
\end{center}
\end{figure}

Another way to understand virtual diagrams is to regard them as representatives for oriented Gauss codes \cite{GPV}, \cite{VKT,SVKT} 
(Gauss diagrams). Such codes do not always have planar realizations. An attempt to embed such a code in the plane
leads to the production of the virtual crossings. The detour move makes the particular choice of virtual crossings 
irrelevant. {\it Virtual isotopy is the same as the equivalence relation generated on the collection
of oriented Gauss codes by abstract Reidemeister moves on these codes.}  
\bigbreak

Figure~\ref{Figure 3} illustrates the two {\it forbidden moves}. Neither of these follows from Reidemeister moves plus detour move, and indeed it is not hard to construct examples of virtual knots that are non-trivial, but will become unknotted on the application of one or both of the forbidden moves. The forbidden moves change the structure of the Gauss code and, if desired, must be considered separately from the virtual knot theory proper. 
\bigbreak

\section{Interpretation of Virtuals Links as Stable Classes of Links in  Thickened Surfaces}\label{stableclass}
There is a useful topological interpretation \cite{VKT,DVK} for this virtual theory in terms of embeddings of links
in thickened surfaces.  Regard each 
virtual crossing as a shorthand for a detour of one of the arcs in the crossing through a 1-handle
that has been attached to the 2-sphere of the original diagram.  
By interpreting each virtual crossing in this way, we
obtain an embedding of a collection of circles into a thickened surface  $S_{g} \times R$ where $g$ is the 
number of virtual crossings in the original diagram $L$, $S_{g}$ is a compact oriented surface of genus $g$
and $R$ denotes the real line.  We say that two such surface embeddings are
{\em stably equivalent} if one can be obtained from another by isotopy in the thickened surfaces, 
homeomorphisms of the surfaces and the addition or subtraction of empty handles (i.e. the knot does not go through the handle).

\begin{figure}
     \begin{center}
     \begin{tabular}{c}
     \includegraphics[width=10cm]{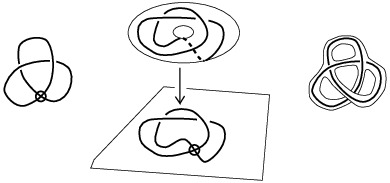}
     \end{tabular}
     \caption{\bf Surfaces and Virtuals}
     \label{Figure 4}
\end{center}
\end{figure}

\begin{figure}
     \begin{center}
     \begin{tabular}{c}
     \includegraphics[width=10cm]{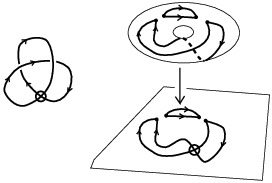}
     \end{tabular}
     \caption{\bf Surfaces and Virtual States}
     \label{Figure 5}
\end{center}
\end{figure}

\vspace{0.5cm}

\noindent We have,
\begin{thm} \cite{Carter,DVK, Kamada3} 
Two virtual link diagrams are isotopic if and only if their corresponding surface embeddings are stably equivalent.
\end{thm}  
\smallbreak
\noindent
\bigbreak  

\noindent In Figure~\ref{Figure 4} and Figure~\ref{Figure 5} we illustrate some points about this association of virtual diagrams and knot and link diagrams on surfaces. Note the projection of the knot diagram on the torus to a diagram in the plane (in the center of the Figure \ref{Figure 4}) has a virtual crossing in the planar diagram where two arcs that do not form a crossing in the thickened surface project to the same point in the plane. In this way, virtual crossings can be regarded as artifacts of projection. The same figure shows a virtual diagram on the left and an \say{abstract knot diagram} \cite{Kamada3,Carter} on the right. The abstract knot diagram is a realization of the knot on the left in a thickened surface with boundary and it is obtained by making a neighborhood of the virtual diagram that resolves the virtual crossing into arcs that travel on separate bands. The virtual crossing appears as an artifact of the projection of this surface to the plane. The reader will find more information about this correspondence  in other papers by Kauffman \cite{VKT,DKT} and in the literature of virtual knot theory.
 
\section{Khovanov and Rasmussen Invariants of Knots in Projective Space}\label{Khovanov}
Dye, Kaestner and Kauffman in \cite{DKK}, have discussed the structure of Khovanov Homology as defined by Manturov \cite{M} for virtual knots and links and they have extended this structure to a version of Lee Homology and defined a Rasmussen invariant for virtual knots and links. We can use our mapping, 
$$\pi: \{Links\ in\ projective\ space/ambient\ isotopy\}\to PVK$$
constructed in Section 2, from projective links to virtual links to define Khovanov homology and a Rasmussen invariant for projective links by taking, for a projective link $K$, the Khovanov homology of $\pi(K)$ and the Rasmussen invariant of $\pi(K).$

\begin{figure*}
\includegraphics[scale=1.5]{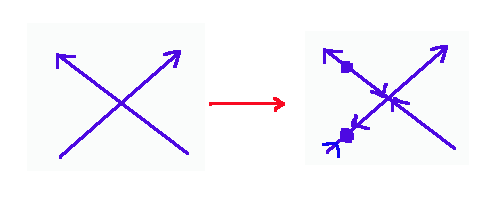}
\caption{An oriented crossing and the corresponding source-sink orientation. }
\label{sosink}
\end{figure*}

\begin{figure}
\includegraphics[scale=1.5]{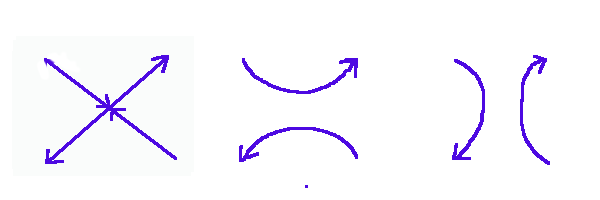}
\caption{The orientations induced on the resolved arcs.}
\label{sosinkres}
\end{figure}

\par We now will give a brief discription of the homology theory. Given a diagram of a projective knot $K$, we may construct the diagram for $\pi(K)$. Then by giving an arbitrary numbering to the crossings, of $\pi(K)$ except the virtual crossings, we can obtain the cube of resolutions. We take the convention, that all $A$ states are arranged first and all $B$ states come at last. That is, each edge corresponds to a crossing switching its $A$ smoothing to a $B$ smoothing. Each resolution consists of a collection of unknotted circles possibly with virtual crossings. To each such circle, we associate a rational vector space $V$ of dimension two, with an abstract basis $\{1, X\}$. This vector space is  a turned into a Frobenius algebra by the multiplication $m$, comultiplication $\Delta$, unit map $\nu$ and  counit map $\epsilon$. These are defined as follows,
\begin{align*}
m: V\otimes V&\to V;\\
m(1\otimes1)&=1;\\
m(1\otimes x)&= m(x\otimes1)=x;\\
m(x\otimes x)&=  0
\end{align*}
\begin{align*}
\Delta: V\to V\otimes V;\\
\Delta(1)&= 1\otimes x + x\otimes 1;\\
\Delta(x)&= x\otimes x
\end{align*}

\begin{align*}
\nu: \mathbb{Q}\to V;\\
&\nu(1)= 1
\end{align*}

\begin{align*}
\epsilon:V \to \mathbb{Q};\\
\epsilon(1)&=1;\\
\epsilon(x)&=0 
\end{align*}

\par To every vertex of the cube, which is a state consisting of several state circles, we associate a module which is the tensor product of $n$ number of copies of $V$, where $n$ is the number of state circles. For each edge which represents the change of the some crossing switching from $A$ smoothing to a $B$ smoothing, we associate one of the three possible maps between the modules on its vertices. If it is a $2-1$ bifurcation we use $m$, if its is a $1-2$ bifurcation we use $\Delta$ and if it is a $1-1$ bifurcation, we use the zero map denoted by $\eta$.\\

\par In the classical case, none of the edges of the cube will correspond to a 1-1 bifurcation and the only maps appearing on the edges are $m$ and $\Delta$'s. The signs assigned to these maps by using the ordering of crossings makes the squares anti-commute. Thus we get a cochain complex and a cohomology theory.  \\

\begin{figure*}
\includegraphics[scale=0.6]{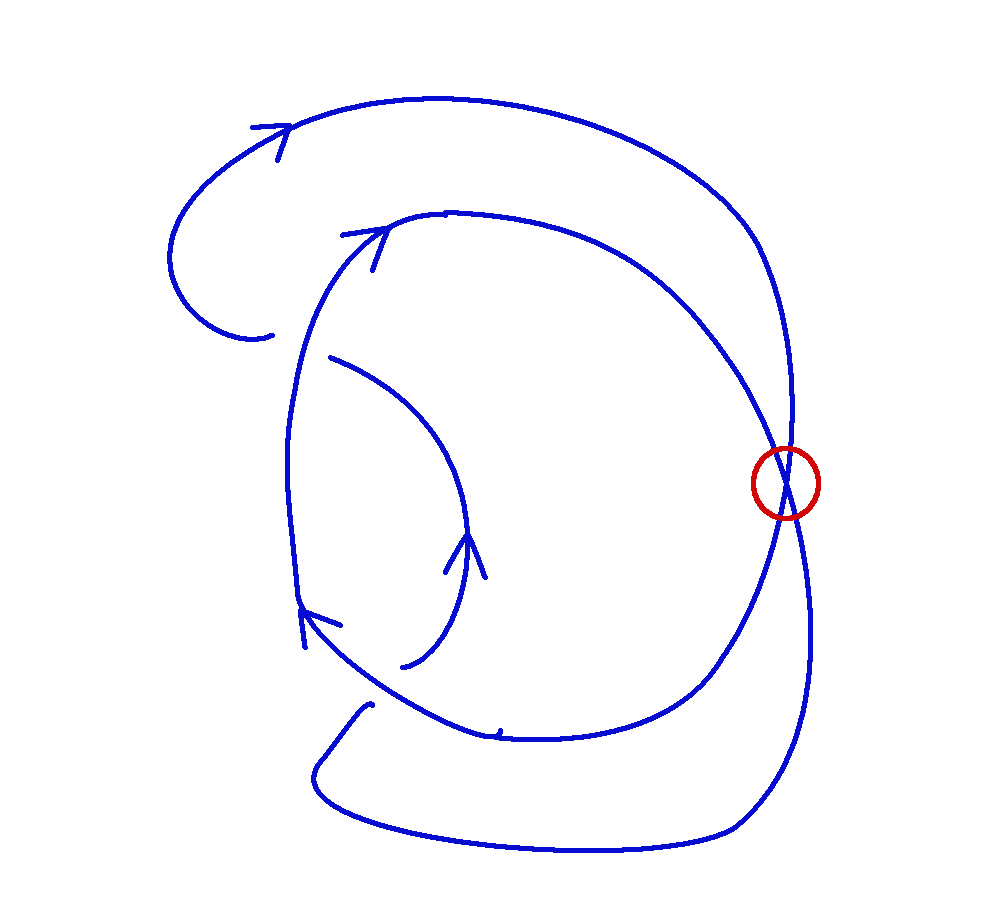}
\caption{An oriented virtual diagram corresponding to an affine unknot.}
\label{probknot}
\end{figure*}

\par In a general projective knot $K$, the diagram of $\pi(K)$ may have virtual crossings. Thus it is possible to have a $1-1$ bifurcations on some edges as in the example shown in Figure \ref{probknot}, for which the cube is shown in Figure \ref{probsquare}. To make all the squares commute, we will use the method of Manturov \cite{M} which is re-interpreted by Dye, Kaestner and Kauffman \cite{DKK}. The basic idea \cite{DKK, M} is as follows. \\

\begin{figure}
\centering
\includegraphics[scale=0.6]{probknot.png}\\
\includegraphics[scale=0.8]{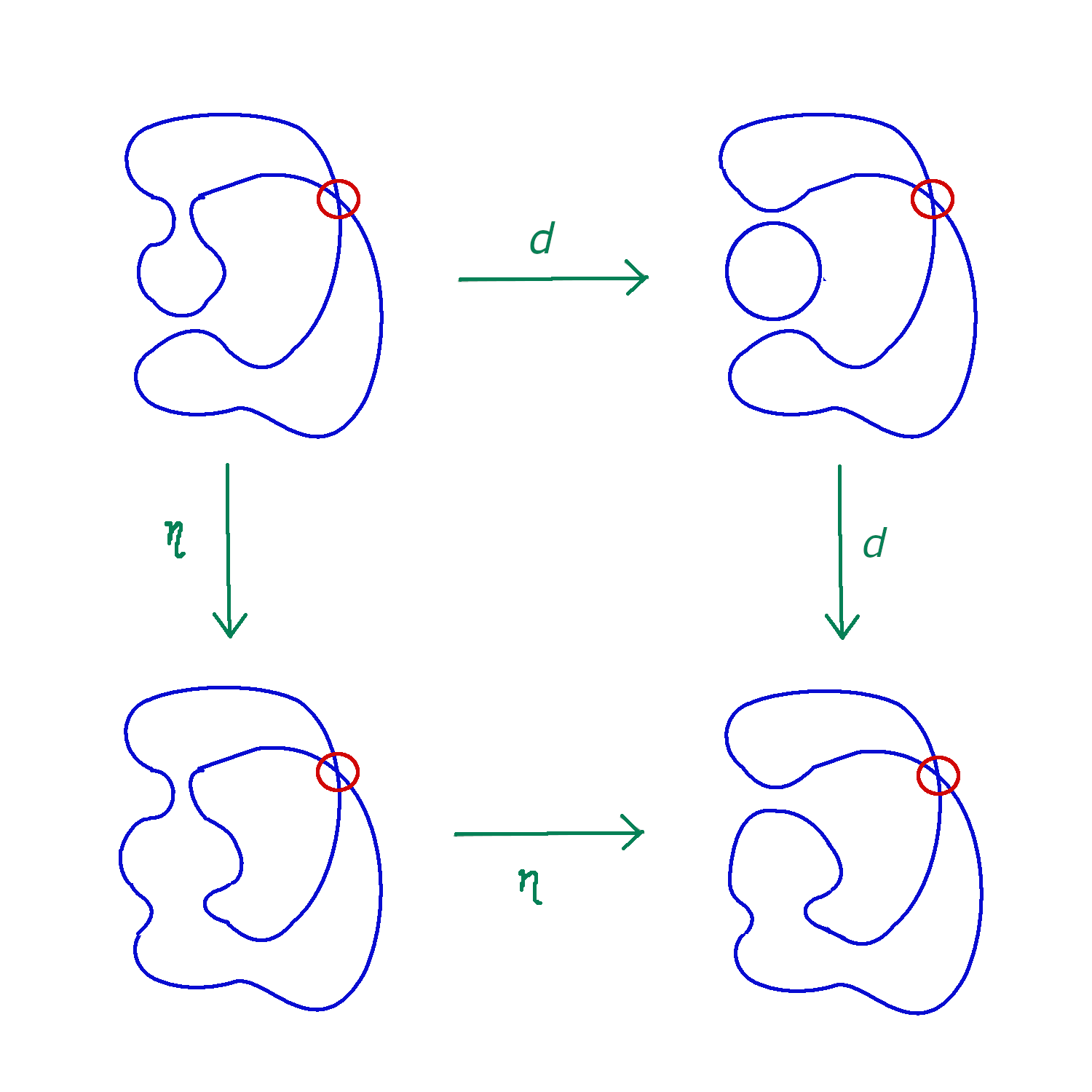}
\caption{An example of a 1-1 bifurcation. On the one hand $\eta\circ\eta=0$, so to make $d\circ d=0$ on the other side, there should be signs.}
\label{probsquare}
\end{figure}

\par Given a knot diagram, we choose an arbitrary orientation on it. See Figure \ref{probknot}. This induces an orientation at each crossing. At each crossing, we give a \say{source-sink} orientation corresponding to the induced local orientation by the scheme shown in Figure \ref{sosink}. The dots on the arcs, are called \say{cut points}, which denote the mismatch of the orientations on both sides. The source-sink orientations on crossings will induce orientations on the resolved arcs as in Figure \ref{sosinkres}. Hence in every state circle we may get multiple arcs with different orientations separated by cut points. We will mark a base point on one of these arcs on every circle. Refer to Figure \ref{basepoint}. The orientation at the base point on every state circle is chosen as a global orientation for that circle.\\

\begin{figure}
\includegraphics[scale=1]{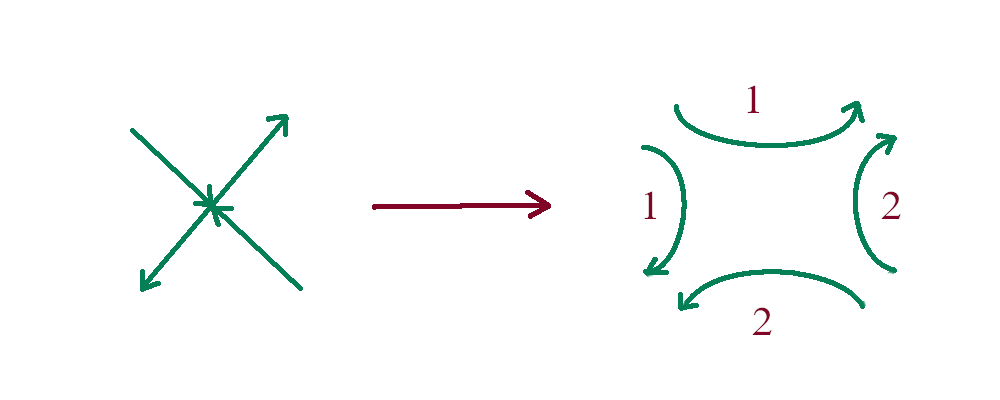}
\caption{The prescription for ordering the arcs of states.}
\label{permute}
\end{figure}

\par Now an arbitrary numbering is chosen on the state circles of each state. As a smoothing is bifurcating, the numbering on the states before and after can be compared. At the states before and after, a permutation of the circles where the circles involved in the bifurcation is promoted to the first position is considered according to the conventions indicated in Figure \ref{permute}. The signs of these permutations at each ends are added to the algebra elements on both ends of the differential. \\

\par Every time an algebra element passes through a cut point, it gets \say{barred}. Notice that, to bar and element in the Frobenius algebra, means if it is an $X$ we multiply it with $-1$, if it is $1$ then we leave it as it is and extend these operations linearly to a general element in the algebra. That is, $\forall a,b\in \mathbb{Q}$, we define:\\

\begin{equation*}
\overline{a1+bX}= a1-bX.
\end{equation*}

\par The algebra element on a state circle should be thought of as placed at the base point. Whenever the differential is applied the algebra elements at the base point(s) before is transported to the site, and the newly produced elements at the site will be transported to the new base point(s). The transports are always in the direction decided by the orientation of each of the state circles. direction of the global orientation on each circle. Before applying the differential, we bar the elements of the algebra on every circle $p$ times, where $p$ is the number of cutpoints it passes through while reaching the site from the base point before the bifurcation. Similarly the newly produced algebra elements will get barred on the path from the concerned site to the base point each time it passes a cut point. And for the algebra elements on the circles, which are not taking part in the bifurcation, they will be transported to the new base point after bifurcation from the initial base point and barred accordingly. \\

\par With this sign convention and the standard conventions in classic Khovanov homology, $d$ becomes a differential. That is, $d\circ d=0$ and we have a graded cohomology theory for all the  $\pi(K)$. 

\begin{prop}\label{khoeq}
If two projective knots $K$ and $K'$ are isotopic, the cohomology groups of $\pi(K)$ and $\pi(K')$ are isomorphic. 
\end{prop}

\begin{proof} As proved by Drobotukhina in \cite{julia}, the diagrams of $K$ and $K'$ are related by a sequence of classical Reidemeister moves and two slide moves. It must be clear that if two diagrams are related by a classical Reidemeister move, then the corresponding chain complexes of their virtual images are related by the standard chain maps which induces isomorphisms on cohomology. Similarly, if two diagrams are related by the \say{sliding a maxima} move, then the corresponding virtual images are related by a virtual Reidemeister move and hence the there is a chain map between their complexes inducing isomorphisms on cohomology given by the construction of virtual Khovanov homology. And finally, from the proof of Theorem \ref{flypeinv}, it should be clear that if two diagrams are related by the \say{sliding a crossing} move, then we get a chain map inducing isomorphism on cohomology from the standard maps in virtual Khovanov homology. Thus the cohomology of $\pi(K)$ only depends on the isotopy type of $K$. 
\end{proof}

\begin{figure}
\includegraphics[scale=0.6]{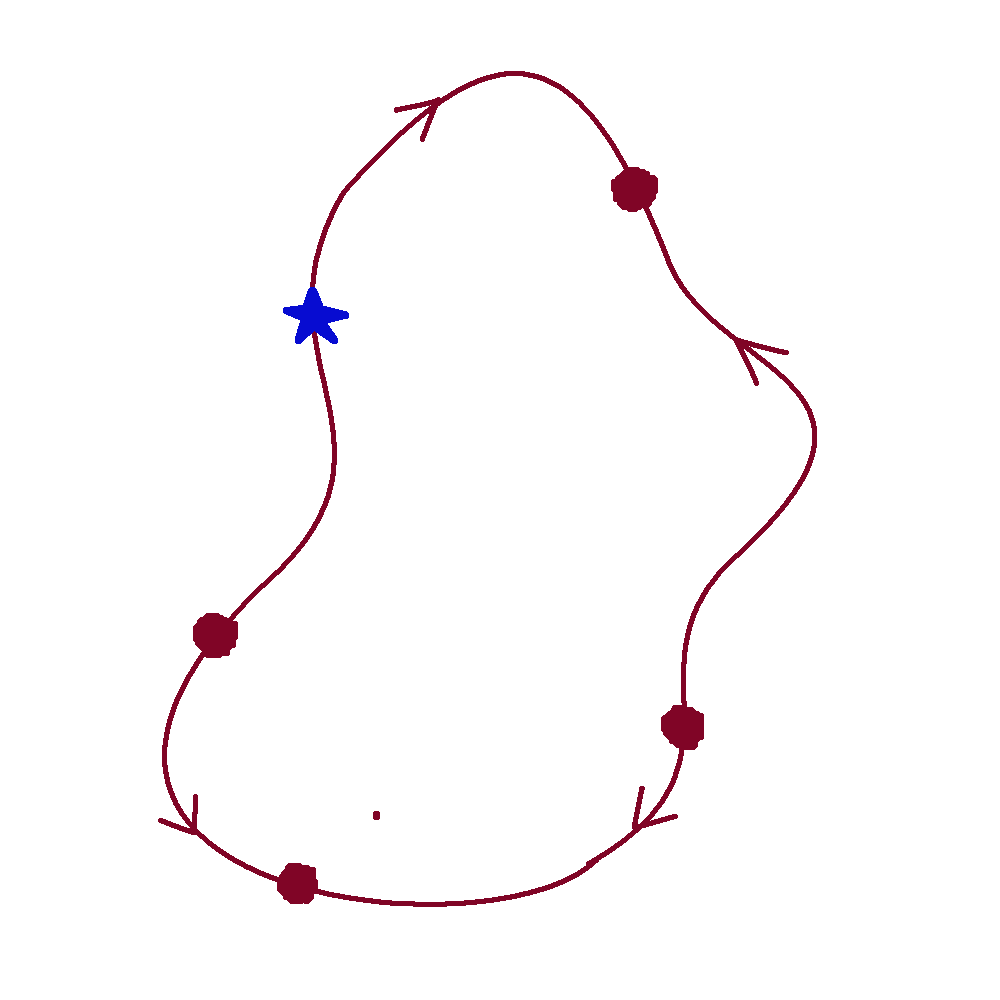}
\caption{A typical state circle. The star denotes a base point. The orientation of this arc will be chosen as the direction in which the elements in the algebra flows through this state circle. }
\label{basepoint}
\end{figure}

We define, \\
\begin{equation*}
Kh(K) := Kh_{DKK}(\pi(K))
\end{equation*}

A major theorem of DKK (Theorem \ref{dkkthm}) which is used in the rest of the paper uses Rasmussen invariant for virtual links for its proof. To do this in projective situation we need to have a definition of Rasmussen invariant $s(K)$, of a projective knot or link $K$. Notice that the Lee deformation extended naturally to the Khovanov complex of a virtual knot is compatible with all the standard chain maps induced by classical and virtual Reidemeister moves. See \cite{DKK} for more details. As shown in the proof of Proposition \ref{khoeq} any slide move performed on a projective diagram induce chain maps on the virtual Khovanov complex which can be deformed as a composition of maps induced by Reidemeister moves. Thus in the same spirit as above we define;\\
\begin{equation*}
s(K) := s_{DKK}(\pi(K)).
\end{equation*}

\par  In order to obtain a result about projective links from this result, we need to discuss cobordism of projective links and how it is related to corresponding cobordisms of the virtual links in the image of the mapping $\pi$. Two projective links $K$ and $K'$ are said to be \textit{cobordant}, if there is a smoothly embedded surface $\Sigma$ in $\mathbb{R}P^3\times I$, whose boundary is entirely inside the boundary of $\mathbb{R}P^3\times I$ so that $\partial \Sigma \cap (\mathbb{R}P^3\times \{0\}) = K$ and $\partial \Sigma \cap (\mathbb{R}P^3\times \{1\}) = K'$. It is clear from Morse theory, that any such surface can be constructed from the fundamental parts: birth, death and saddles shown in Figure \ref{BDS}. The diagrams of links at two ends of each of these surfaces are easy to see. Birth and death will correspond to appearance and disappearance of a circle unlinked with the rest of the diagram. Similarly saddles correspond to two oppositely oriented arcs changing from one smoothing to the other. See Figure \ref{BDS}. Hence this notion of cobordism can be explained combinatorially at the diagram level. We say that two projective link diagrams are {\it cobordant} if one can be obtained from the other by a sequence of oriented saddle replacements, births and deaths of circles and isotopy of projective links. See Figure~\ref{BDS}. The result of such a sequence of changes in the diagram can be regarded formally as the generation of an orientable surface whose boundary consists in the first and second diagrams. If diagrams $K$ and $K'$ are cobordant, then the genus of the cobordism surface is called the {genus of the cobordism between $K$ and $K'.$} These notions apply to classical diagrams, to diagrams of links in $RP^{2}$ that represent projective links in $RP^{3},$ and to virtual link diagrams. In the classical and virtual cases, every link is cobordant to the empty link. That is, for a given diagram $K$ there is a sequence of saddles, births and deaths that results in an empty diagram. Correspondingly, there is a surface whose boundary is the link $K$. In the projective case, this is true only for null homologous links. In the case of a class-1 link we can always construct a cobordism to an unknotted loop that represents a projective line in $RP^{2}$ diagrammatically the class-1 unknot in $RP^3$ in terms of the corresponding embedding in the three dimensional projective space. We call these surfaces the {\it 4-ball surfaces for $K$} uniformly where this terminology does in fact indicate a surface in the four-ball in the classical case. We call the least possible genus among all such surfaces the {\it 4-ball genus $g_{4}(K)$.} \\

\par In Figure~\ref{CSS} we illustrate the construction of the Seifert surface for a classical trefoil knot. In Figure~\ref{VCS} we show how the Seifert construction can be seen as cobordism by performing saddle points near each crossing. In Figure~\ref{VSS} we show how this cobordism analogue for the Seifert construction can be used to produce a virtual Seifert surface for a virtual knot.\\

The main theorem of \cite{DKK} on virtual knot cobordism using Khovanov homology for virtuals is,\\ 

\begin{thm}\label{dkkthm} The virtual 4-ball genus of a positive virtual link $L$ is equal to its virtual Seifert genus. The virtual Seifert genus has the specific formula $g = \frac{ (C+1 - S)}{2}$ where $C$ is the number of classical crossings in $L$ and $S$ is the number of Seifert circuits, obtained by smoothing all the crossings in an oriented fashion. 
\end{thm}

\begin{thm} \label{ourthm}
The 4-ball genus of a positive projective knot is equal to its Seifert genus. 
\end{thm}

\begin{proof} 
From a link $K$ in $RP^{3}$ by applying our mapping $\pi$ we obtain a virtual link $\pi(K)$. It is then the case that if $K$ and $K'$ are cobordant projective links, then $\pi(K)$ and $\pi(K')$ are cobordant virtual links.$\ K$ is positive if and only if $\pi(K)$ is positive. Hence the theorem follows from applying Theorem \ref{dkkthm}. 
\end{proof}

If $K$ is a projective link which has an odd number of class-1 components (a class-1 link) then there is a cobordism that will always take $K$ to a class-1 unknot. Call this $\Omega$. Then $\pi(\Omega)$ is an unknotted virtual loop and hence bounds a disk. This means that cobordisms of class-1 projective links differ from the cobordisms we can apply to their virtual counterparts. But we can define the genus of a class-1 link $K$ to be the minimal genus of a surface that is a cobordism of $K$ to $\Omega.$ Since $\pi(\Omega)$ bounds a disk in the virtual context, we see that, $g_{4}(\pi(K))$ is equal to this minimal genus for $K$ in projective space.\\

For a positive non-affine knot $K$ that is not class-1, Theorem \ref{ourthm} applies to give $g_{4}(K)$. For example, see Figure~\ref{PG} where we illustrate such a knot in the lower portion of the figure and determine that $g=g_4 (K)= 1$ by finding that the virtual version of $K$ has $C=2, S=1$ and applying the formula $g = \frac{(C+1 - S)}{2} = \frac{(2+1 - 1)}{2} = 1.$\\

\begin{figure}
     \begin{center}
     \begin{tabular}{c}
     \includegraphics[width=6cm]{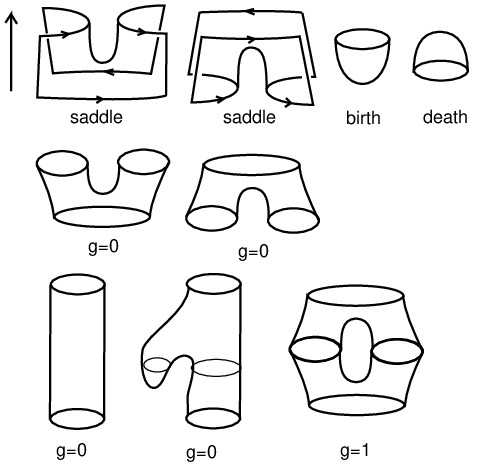}
     \end{tabular}
     \caption{\bf Cobordism of Diagrams}
     \label{BDS}
\end{center}
\end{figure}

\begin{figure}
     \begin{center}
     \begin{tabular}{c}
     \includegraphics[width=6cm]{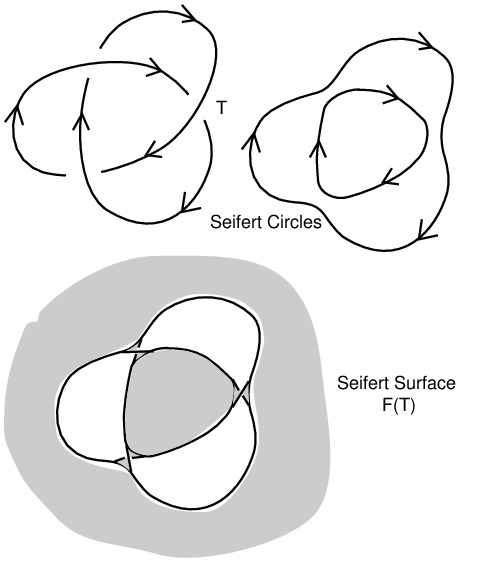}
     \end{tabular}
     \caption{\bf Classical Seifert Surface}
     \label{CSS}
\end{center}
\end{figure}

\begin{figure}
     \begin{center}
     \begin{tabular}{c}
     \includegraphics[width=6cm]{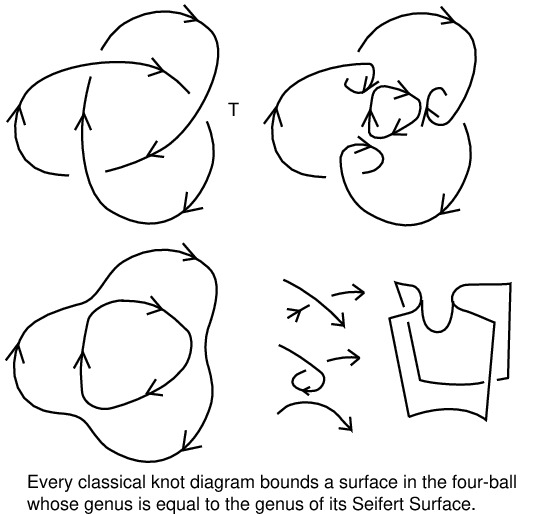}
     \end{tabular}
     \caption{\bf Virtual Cobordism Surface}
     \label{VCS}
\end{center}
\end{figure}

\begin{figure}
     \begin{center}
     \begin{tabular}{c}
     \includegraphics[width=6cm]{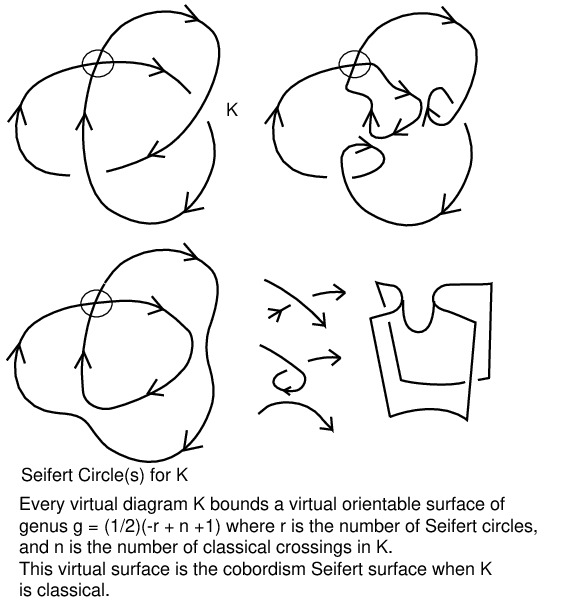}
     \end{tabular}
     \caption{\bf Virtual Seifert Surface}
     \label{VSS}
\end{center}
\end{figure}

\begin{figure}
     \begin{center}
     \includegraphics[width=8cm]{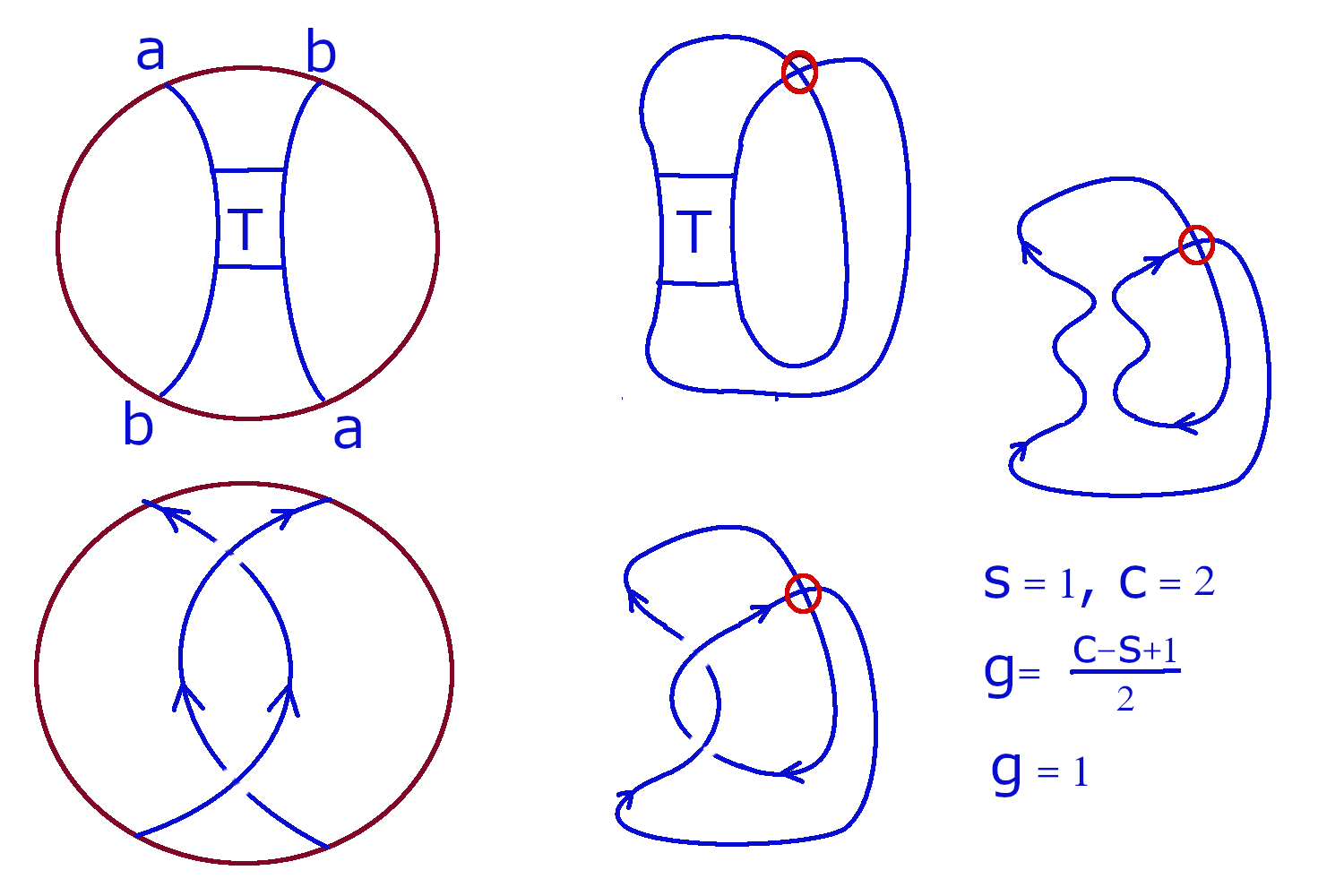}
     \caption{\bf Finding Projective Genus}
     \label{PG}
\end{center}
\end{figure}

\begin{figure}
\includegraphics[scale=1]{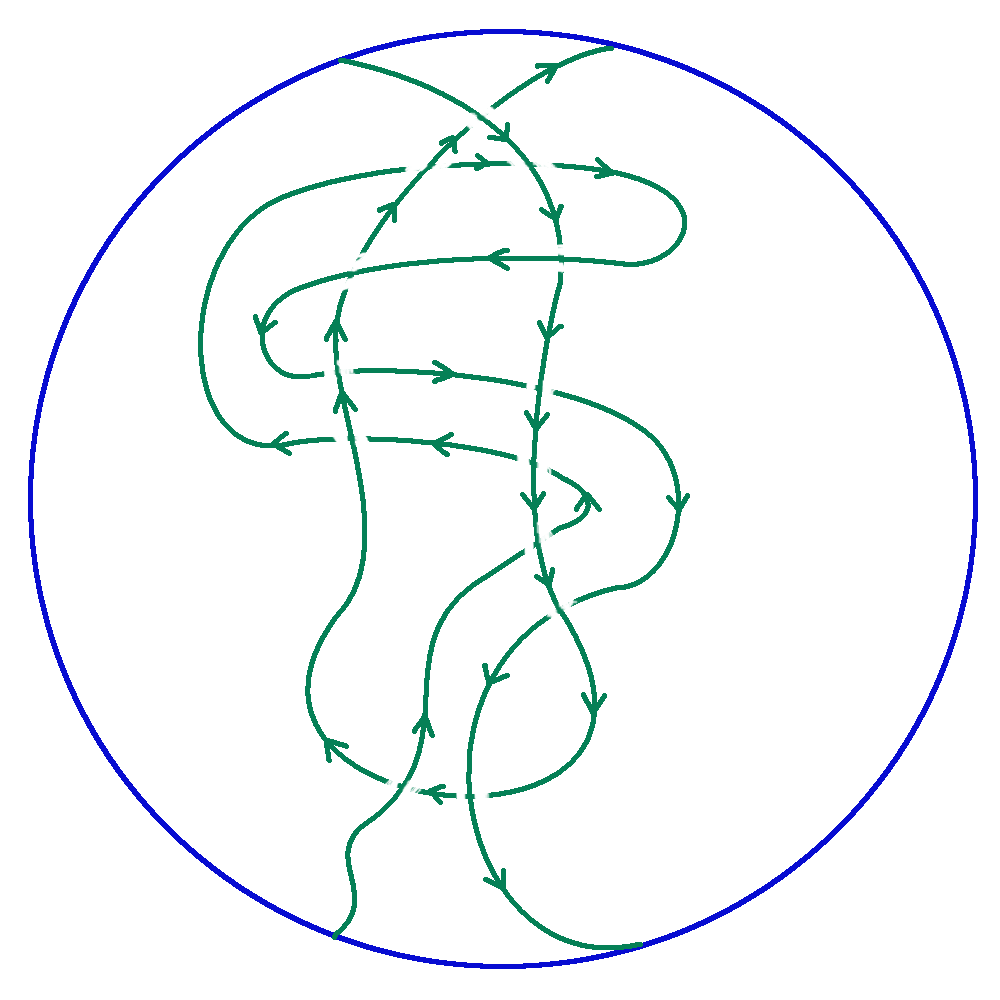}
\caption{A projective ribbon knot with $\frac{C+1-S}{2}=\frac{13+1-6}{2}=4$. }
\label{projribknot}
\end{figure}

\par It is to be noted that this formula works only for positive knots. The knot $R$ shown in Figure \ref{projribknot}, has the following Jones polynomial,
\begin{equation*}
f_R = 2+A^{-18}+A^{-16}-A^{-14}-2A^{-12}+2A^{-8}-2A^{-4}-A^{-2}+A^2,
\end{equation*}
which has monomials with degree not congruent to 4. Hence by Corollary \ref{drobcor} it is non-affine.  See the appendix for the calculation of bracket polynomial. $R$ is slice, since it is ribbon and hence it will have 4-ball genus 1. But $\frac{C+1-S}{2}=\frac{13+1-6}{2}=4$. \\

Ciprian Manolescu and Michael Willis  defined Khovanov Homology and a Rasmussen invariant \cite{Manolescu} by modifying a Khovanov homology theory for projective knots defined by Bostjan Gabrovsek \cite{gab}.
Their definitions are not the same as ours, but they obtain very similar results about positive links. A comparison of the two approaches is called for and we will discuss this comparison below and translate between the two points of view.

\subsection{A comparison with \cite{Manolescu}}\label{comparison}

\par For brevity of notation, in what follows, we will denote the cohomology theory of Dye, Kaestner and Kauffman as DKK theory and the cohomology theory of Manolescu and Willis as the MW theory.\\

\begin{figure}
\includegraphics[scale=1.5]{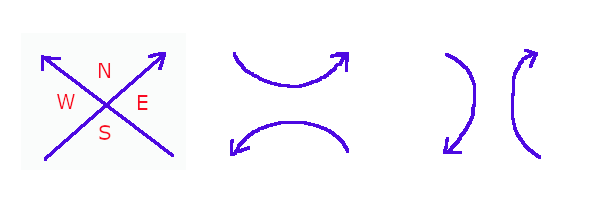}
\caption{Locally consistent orientations given by the cardinal directions.}
\label{carddirs}
\end{figure}

\par In the MW-theory, they introduce a new Frobenius algebra for the class-1 circles in a state, generated by $1$ and $ \overline{X}$, which is isomorphic to the Frobenius algebra above generated by $1$ and $X$. A word of caution is to be mentioned that the $\overline{X}$ in MW theory and the $\overline{X}$ above obtained from barring operation from DKK theory are entirely different. In MW theory, the introduction of this new algebra means that the final Khovanov cohomology has knowledge of the singular homology class represented by the link.\\

\par In the MW theory, one has to choose a global orientation on the link diagram and then  arbitrary local orientations at each of the crossings. Each of the state circles are  given a numbering and global orientation arbitrarily. The chosen local orientations at a crossing determine cardinal directions on the local quarter planes as shown in Figure \ref{carddirs}. The cardinal directions at every crossing determines a set of \say{consistent} orientations on the resolved arcs as in Figure \ref{carddirs}. There are three rules for deciding signs, while applying the differential. These are called  \textit{Permutation rule}, \textit{Nearby consistency rule} and the \textit{Far orientations rule}.\\

\par Permutation rule is exactly same as the rule in the DKK theory, described above, using Figure \ref{permute}. Nearby consistency rule says that if the consistent orientation on an arc at a newly produced site (see Figure \ref{carddirs}), is mismatched with the global orientation, then this determines a sign on the $X$ component of the algebra element produced there. Comparing Figure \ref{sosinkres} and Figure \ref{carddirs}, it is clear the local orientations on the state circles are exactly the same in both the theories. Hence a state in the MW theory may be used to construct a state in the DKK theory by applying the map $\pi$ to the state and change each of the cardinal orientations to source-sink orientations. Similarly a state in DKK theory determines a state in MW theory, as shown in Figure \ref{eqmap}. \\

\begin{figure}
\centering
\includegraphics[scale=1.8]{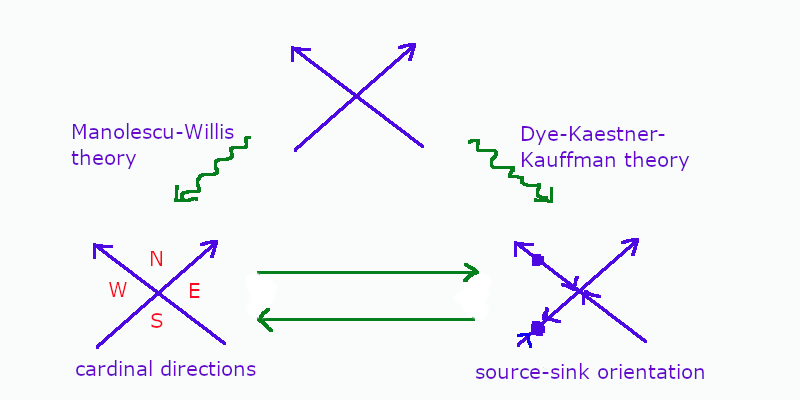}
\caption{Equivalence of our theory with MW theory}
\label{eqmap}
\end{figure}

\par We can use the DKK cutpoint assignment to choose global orientations for each loop in a state. We do this by taking the local orientation at the base point of a circle in a DKK state and then using its direction co-determine a global orientation for the corresponding loop in the MW state. When the differential is applied, the local orientation at the arcs at the site of a loop will be either the same or different from the global orientations according to the parity of number of cutpoints that one must traverse while travelling towards the base point. Now notice that we have chosen to orient the MW state (each loop given an orientation) using the local orientations at the basepoint from the DKK state. When the local orientation at a site mismatches with the global orientation, this corresponds to a sign change for the $X$'s on the MW state by nearby consistency rule. This can occur only when there are odd number of cut points between the base point and the site. Thus the DKK rule involving parity of number of cut points becomes the exact same sign change as in the MW theory. \\

\par The far orientations rule is as follows. After applying the differential, the result is multiplied with a sign corresponding to the parity of the number of circles, which do not pass through the neighbourhood of the bifurcating crossing and are carrying an $X$ and have distinct orientations in the diagrams before and after the bifurcation. Note that the orientation of a far circle flips when the base points chosen before and after are separated by an odd number of cutpoints. Which means, in the new state, the algebra element on that circle has to be barred in the DKK rule. Hence this rule is already implied by the cutpoint rules. \\

\par Thus, in this special set of DKK states determined by MW states, the signs sprinkled by either of the two theories are coherent. So we have an equivalence of the two cohomology theories. Given a projective link diagram $L$, we have natural chain map,

\begin{align*}
\alpha: C_{MW}(L) &\rightarrow C_{DKK}(L)\\
 [s] &\mapsto [\pi(s)]
\end{align*}

\noindent where $C_{DKK}(L)$ represents the chain complex associated as above and $C_{MW}(L)$ is \say{unlabeled} chain complex associated as in \cite{Manolescu}, with $\overline{X}$ replaced by $X$. \say{Unlabeled} means we do not distinguish the class 1 unknot from the class 0 unknot. Our result is that the unlabeled cohomologies are isomorphic.\\

\begin{thm}
The chain map $\alpha$ induces isomorphisms on the cohomology groups.
\end{thm}

\begin{proof}
From the discussion above, it follows that the map $\alpha$ is  a bijection on the chain groups, which commutes with the differentials. Thus it induces isomorphisms on the cohomology groups. 
\end{proof}

\begin{rem}
Our cohomology theory can be modified by adding markings for class-1 state circles so that we get cohomology groups that are identical to those in the MW theory.
\end{rem}

\par In our method, we apply the map $\pi$ to a projective knot diagram, drawn on a disk to obtain a virtual diagram, and then construct the cube of resolutions for it. Instead of this, we can  first construct the cube of resolutions of a projective diagram on the disk where the states will consist of diagrams of disjoint unknots and then apply the map $\pi$ to each of the vertices and obtain the cube of resolutions of the corresponding virtual knot. Note that out of in every state, most of the state circles are class-0 unknots. There can be at most one class-1 unknot in a state. This will be represented by an arc which connects some pair of diametrically opposite points. Before applying the map $\pi$, we may mark this state circle, by putting a diamond on the arc. Now after applying the map $\pi$, in the cube of resolutions of the virtual knot, we have a unique state circle, in those states with the marking. While applying the TQFT, these circles are mapped to a different Frobenius algebra $V^*$, generated by $1^*, x^*$. The structure maps for this algebra are as follows.

\begin{align*}
m: V\otimes V&\to V;\\
m(1^*\otimes1^*)&=1^*;\\
m(1^*\otimes x^*)&= m(x^*\otimes1^*)=x^*;\\
m(x^*\otimes x^*)&=  0
\end{align*}
\begin{align*}
\Delta: V\to V\otimes V;\\
\Delta(1^*)&= 1^*\otimes x^* + x^*\otimes 1^*;\\
\Delta(x^*)&= x^*\otimes x^*
\end{align*}

\begin{align*}
\nu: \mathbb{Q}\to V;\\
&\nu(1^*)= 1^*
\end{align*}

\begin{align*}
\epsilon:V \to \mathbb{Q};\\
\epsilon(1^*)&=1^*;\\
\epsilon(x^*)&=0 
\end{align*}

\begin{figure}
\includegraphics[scale=0.5]{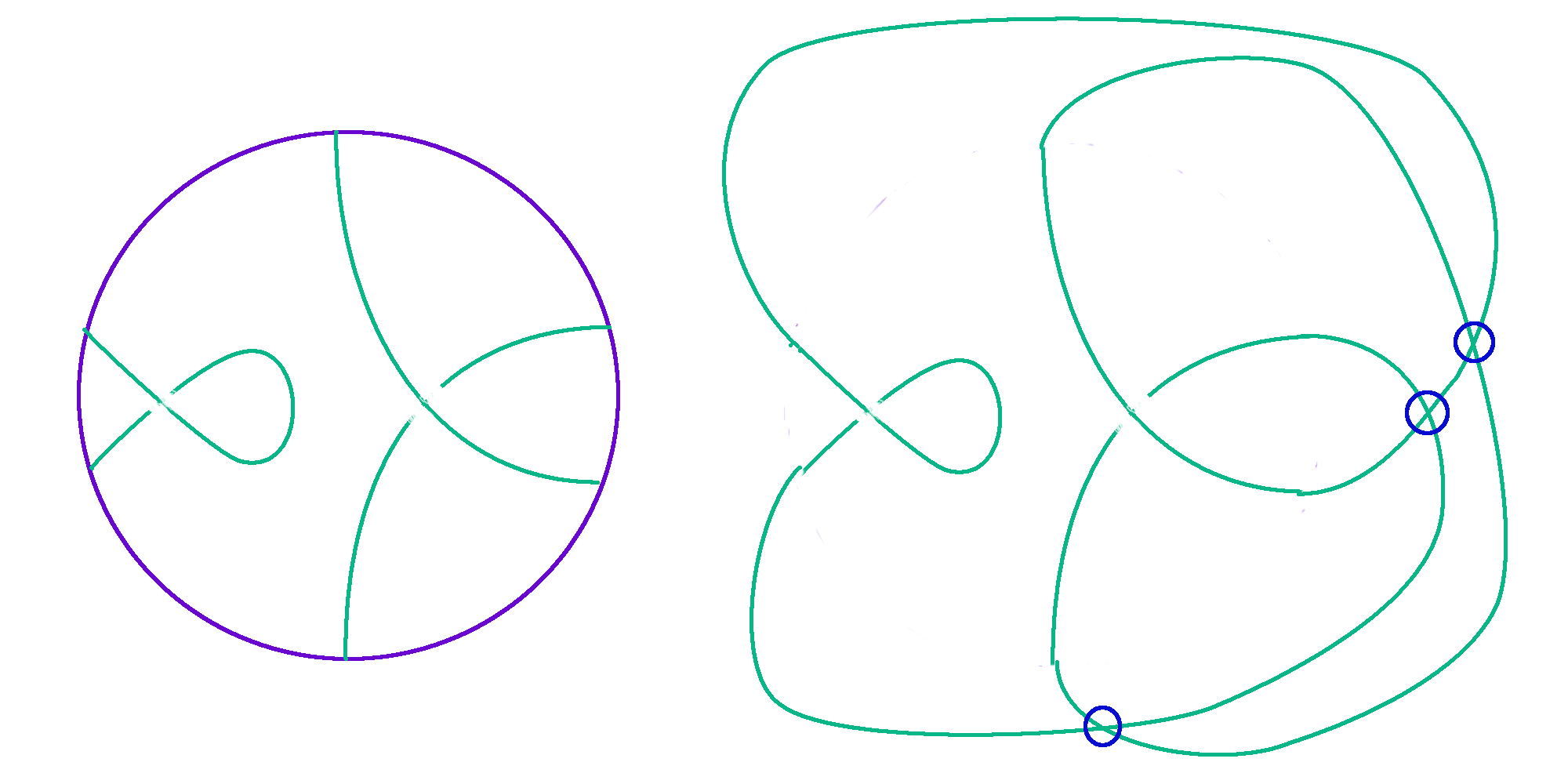}
\caption{A projective knot and the corresponding virtual knot.}
\label{knotcube}
\end{figure}

\begin{figure}
\includegraphics[scale=0.3]{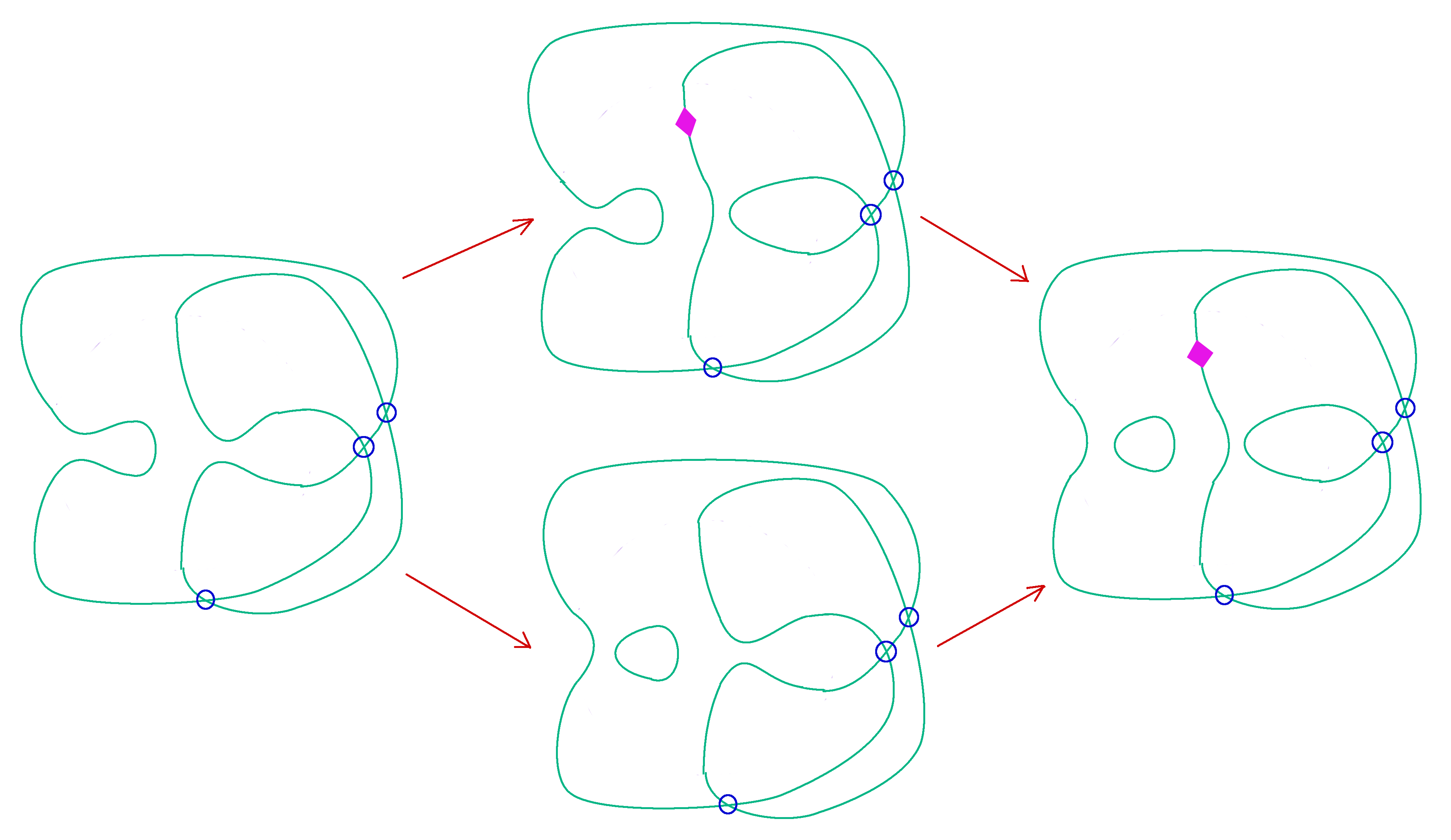}
\caption{The cube of resolutions for a projective virtual knot. The class-1 state circles are marked with a diamond.}
\label{cubeclass1}
\end{figure}

\section{What this method cannot see!}\label{blindness}

\begin{figure}
\centering
\includegraphics[scale=0.8]{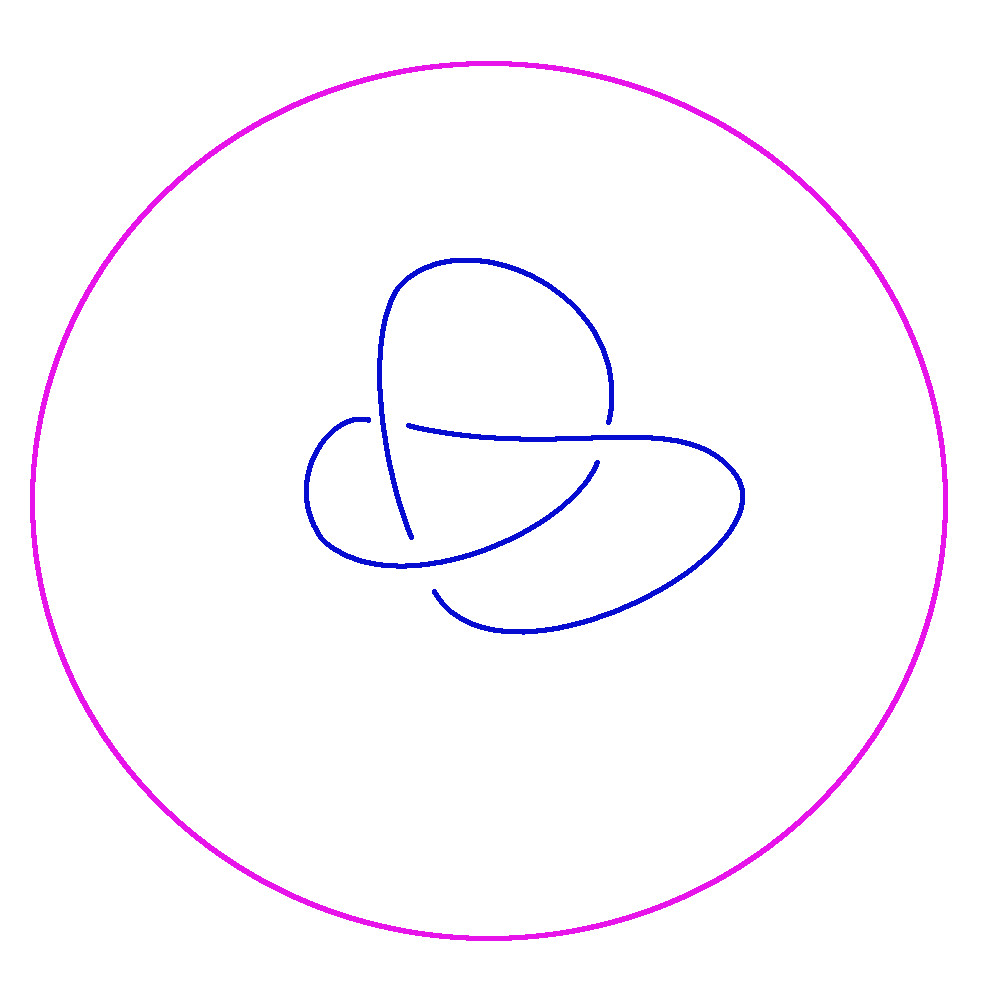}
\caption{Affine trefoil knot}
\label{AK}
\end{figure}

\begin{figure}
\centering
\includegraphics[scale=0.8]{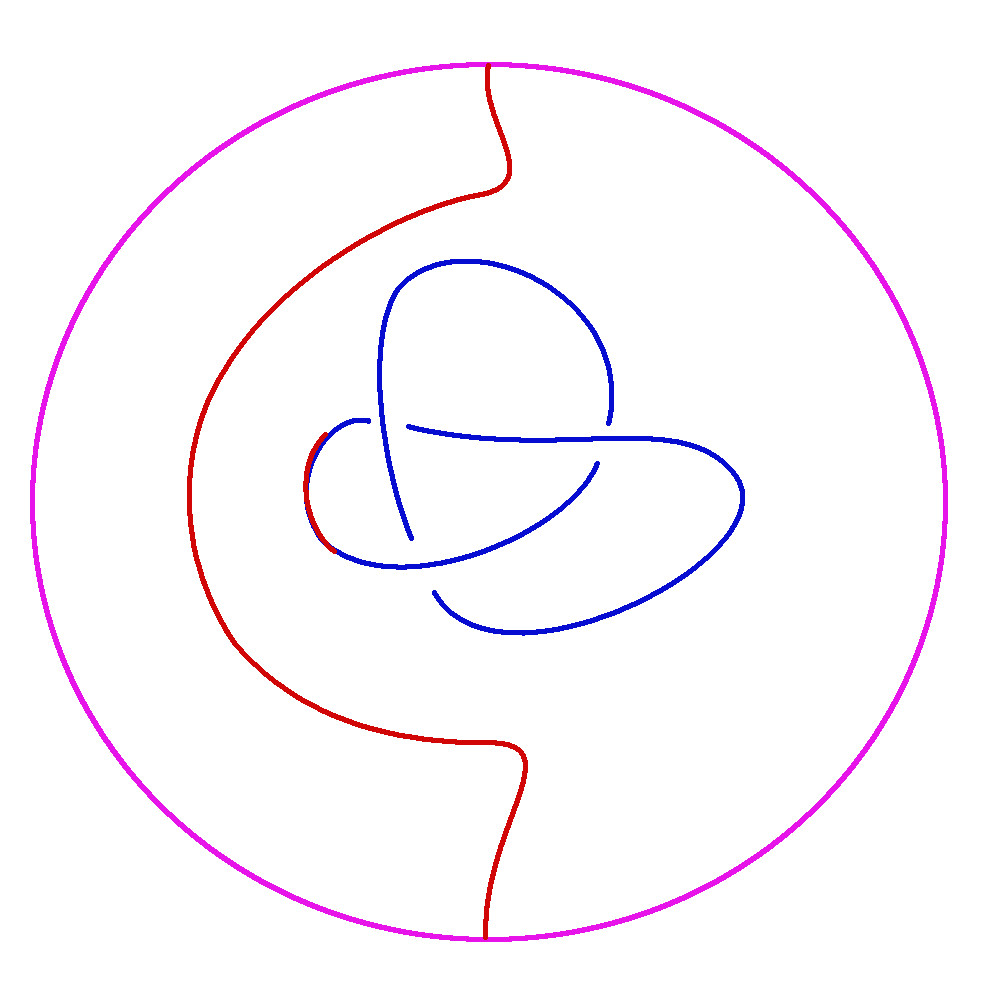}
\caption{Affine trefoil knot and a non-affine un-knot }
\label{AKN}
\end{figure}

\begin{figure}
\centering
\includegraphics[scale=0.8]{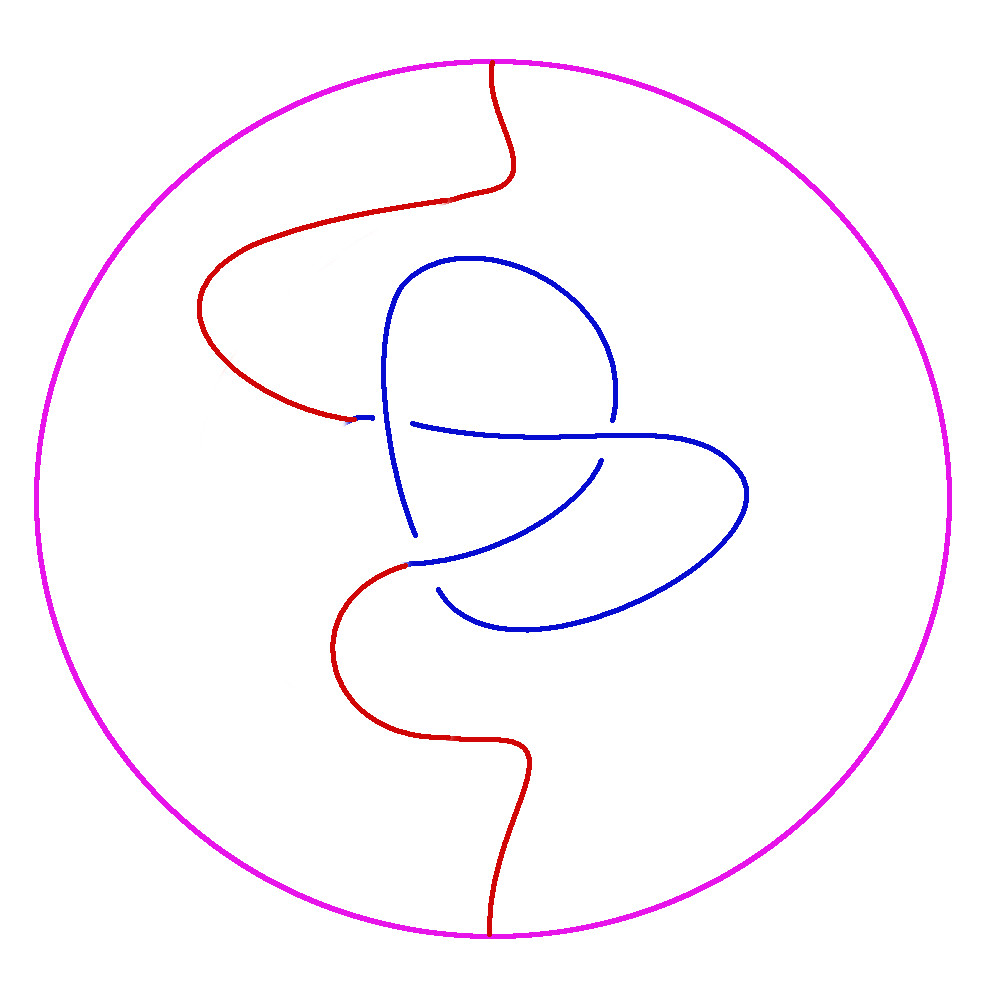}
\caption{Non-affine trefoil knot}
\label{NAK}
\end{figure}

   The association of a virtual knot to a projective knot in this way indeed is not a complete solution to the problem distinguishing projective knots. For example here we discuss an infinite family of pairs of knots which have identical associated virtual knots. \\
\par   Consider the affine trefoil knot as shown in Figure~\ref{AK}. Notice that the virtual knot associated to this is the classical trefoil knot. Choose a non-affine unknot passing nearby this knot as in Figure~\ref{AKN}. We may remove a small arc from the affine knot and the non-affine unknot and join the remaining parts of both of the knots. We will get a knot as shown in Figure~\ref{NAK}. This is a class-1 knot and hence distinct from the knot we started with. This process is called \say{projectivization} in \cite{MN}. Notice that the virtual knot associated to both the projectivized trefoil knot and the affine trefoil knot, by our mapping $\pi$ is the classical trefoil knot. \\

\par An interesting example is the case of the Figure-8 knot. If we projectivize the affine Figure-8 knot on two different regions, we will get two diagrams as shown in Figure~\ref{TU}. It is not clear to the authors whether these two diagrams represent the same knot. Notice that the virtual knot associated to both of these knots, is exactly the classical Figure-8 knot. Hence our method of associating a virtual knot does not help in the problem of distinguishing these knots as projective knots. Thus none of the virtual knot invariants or the Jones polynomial of Drobotukhina \cite{julia} will help us here. 
\begin{figure}
\centering
\includegraphics[scale=0.8]{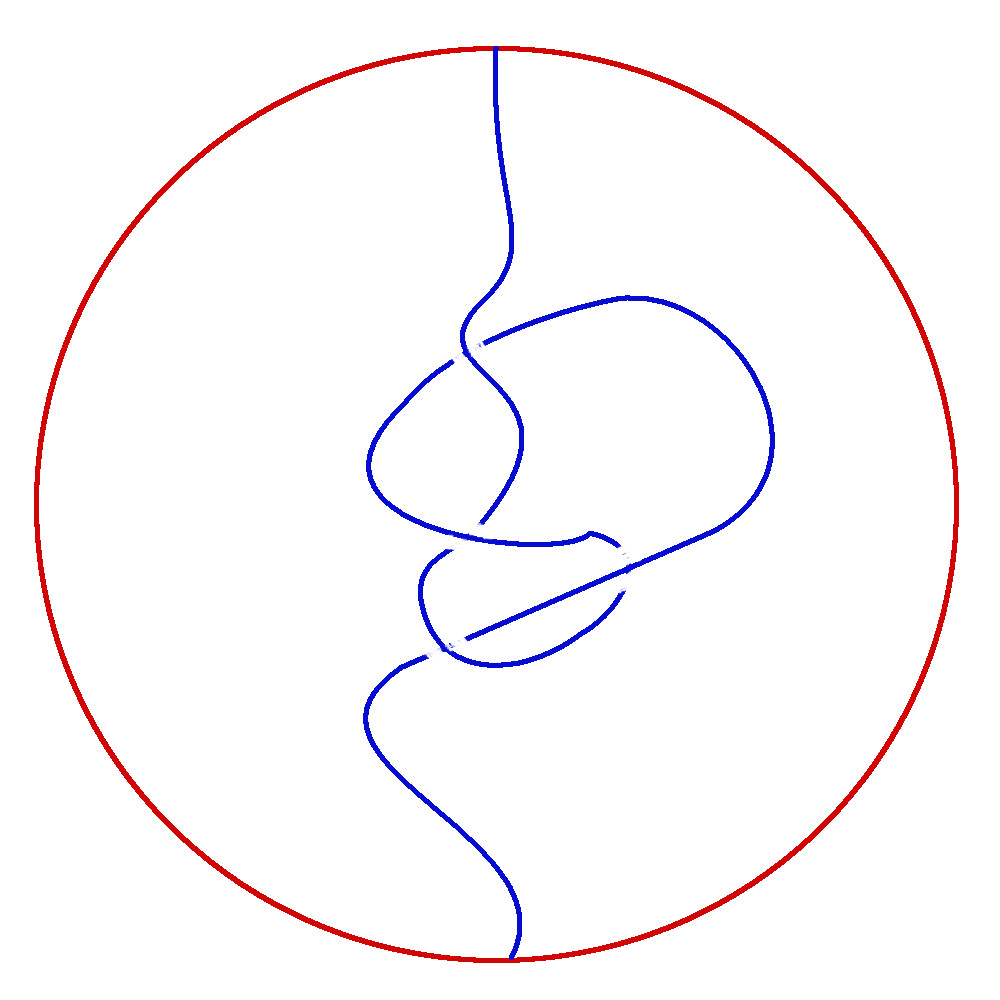}
\includegraphics[scale=0.8]{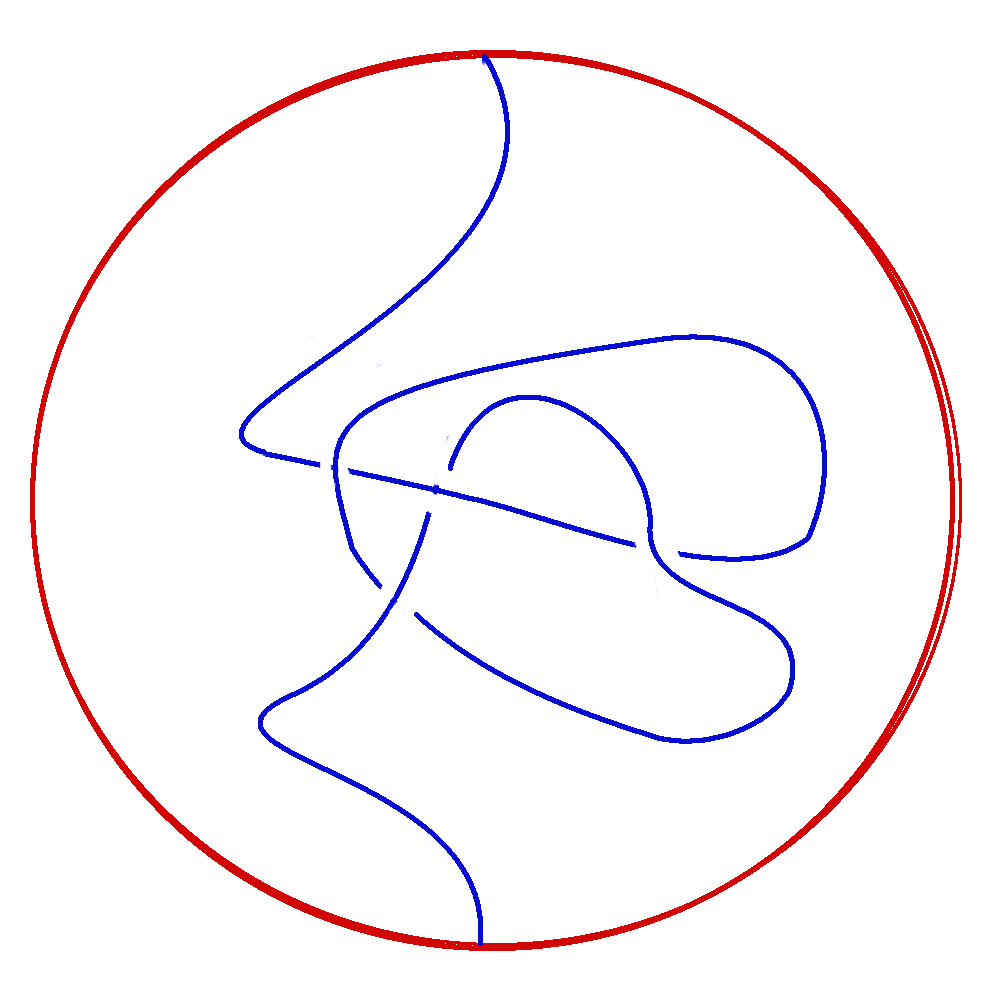}
\caption{Diagrams of two projectivizations of the Figure-8 knot}
\label{TU}
\end{figure}

\newpage

\section{Appendix: Bracket polynomial of the knot in Figure \ref{projribknot}}\label{appbracket}

\begin{figure}
\includegraphics{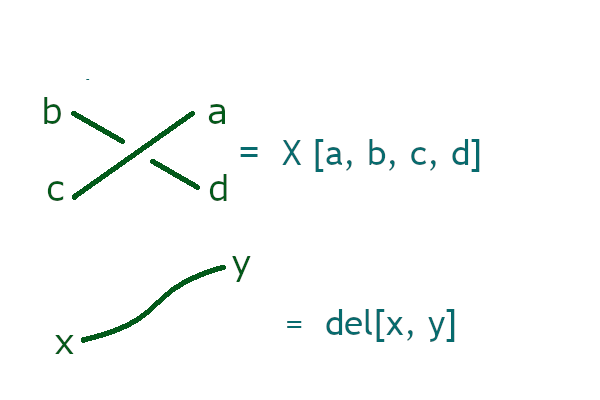}
\caption{Conversion formula to the code.}
\label{rules}
\end{figure}

\begin{figure}
\includegraphics[scale=1]{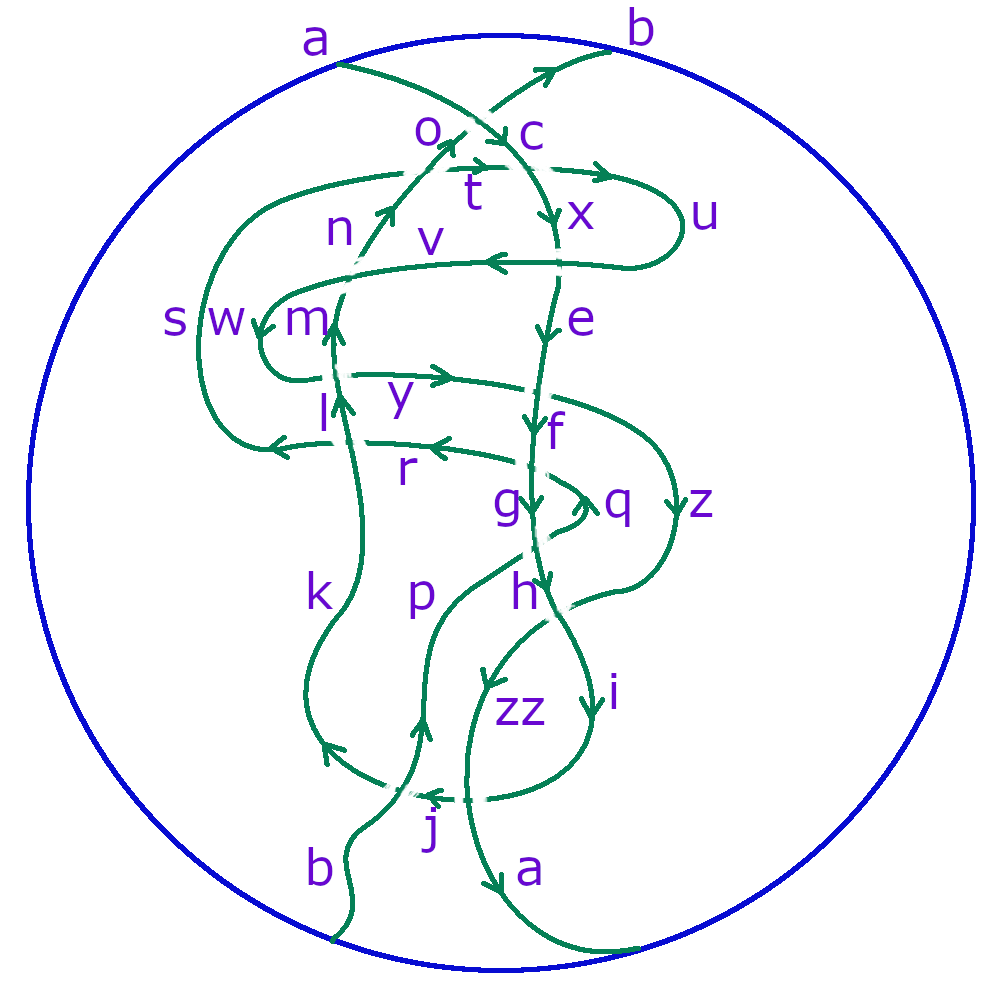}
\caption{Indexing the arcs of the ribbon knot in Figure \ref{projribknot} to calculate bracket polynomial.}
\label{knotindex}
\end{figure}

We may index the arcs of the diagram as in Figure  \ref{knotindex}. The $Mathematica$ code for computing the bracket is given below. The conversion rules are given in Figure \ref{rules}.

\includegraphics[scale=0.7]{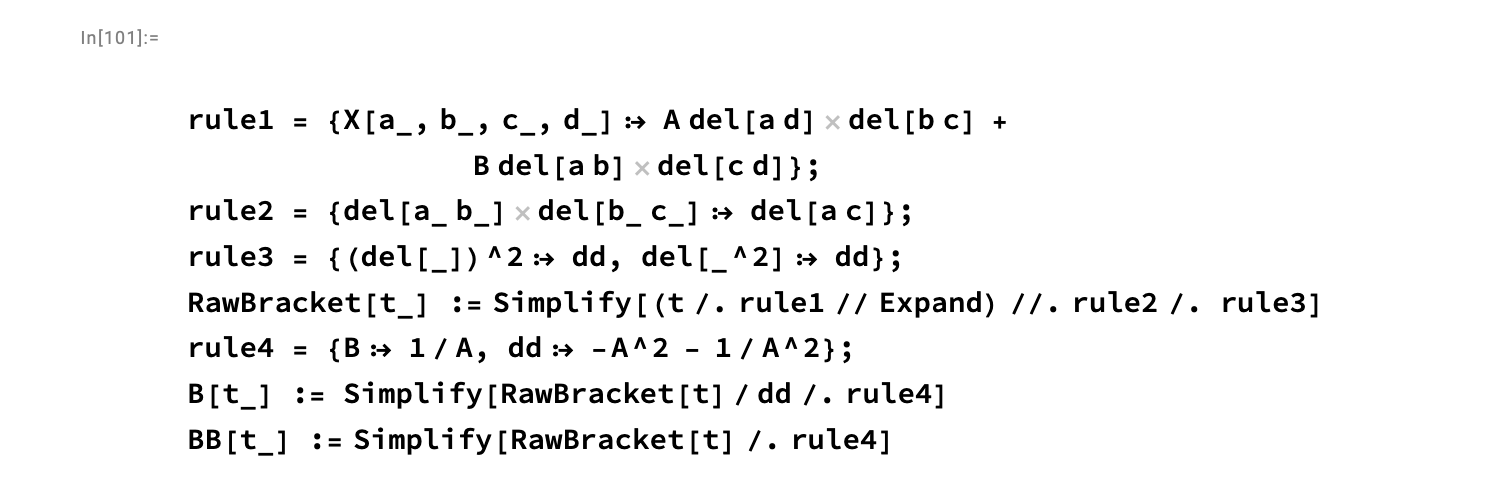}\\

\includegraphics[scale=0.7]{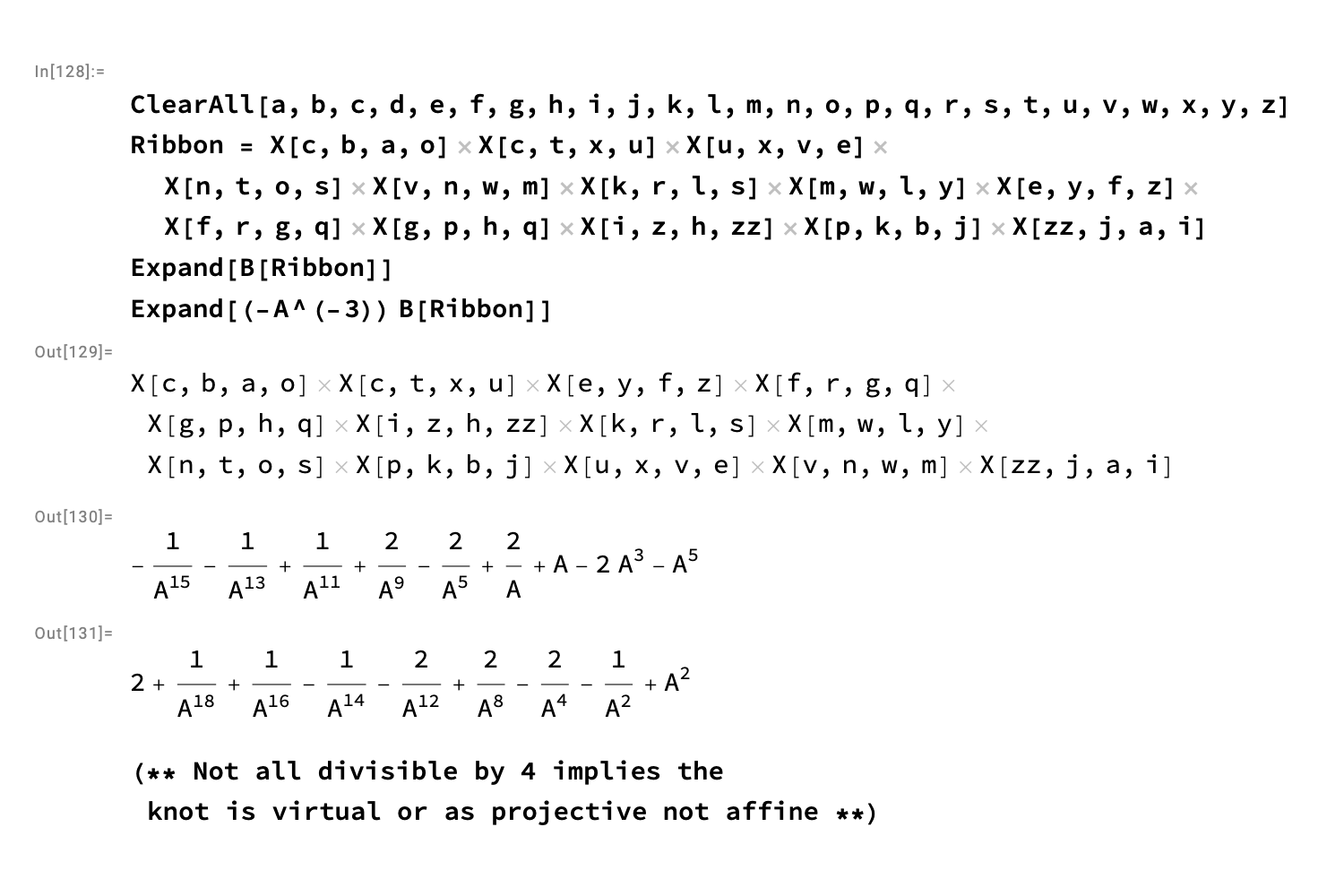}


\vspace{1in}
\begin{center}
Louis H. Kauffman\\ Department of Mathematics, Statistics \\ and Computer Science (m/c
249)    \\ 851 South Morgan Street   \\ University of Illinois at Chicago\\ Chicago, Illinois 60607-7045\\ and \\
International Institute for Sustainability with Knotted Chiral Meta Matter (WPI-SKCM2),\\
Hiroshima University, 1-3-1 Kagamiyama, Higashi-Hiroshima, Hiroshima 739-8526, Japan\\ $<$kauffman@uic.edu$>,$\\
\vspace{0.3in}
Rama Mishra
\\ Indian Institute of Science Education and Research, Pune, India\\
$<$r.mishra@iiserpune.ac.in$>$ \\ \vspace{0.15in} and \\
\vspace{0.15in}
Visakh Narayanan
\\ Indian Institute of Science Education and Research, Mohali, India\\
$<$vishakhme@gmail.com$>$

\end{center}

\end{document}